\documentclass[final,3p]{elsarticle}
\usepackage{amssymb}
\usepackage{tikz}
\usepackage{amsthm}
\usepackage{amsmath}
\usepackage{mathrsfs}
\usepackage{eucal}
\usepackage{float}
\usepackage{hyperref}
\newcommand{\Hom}{\operatorname{Hom}}

\newcommand{\dist}{\operatorname{dist}}

\newcommand{\bbZ}{\mathbb{Z}}
\newcommand{\bbN}{\mathbb{N}}
\newcommand{\bbR}{\mathbb{R}}
\newcommand{\pe}{p^e -1}

\newcommand{\dis}{\operatorname{disp}}
\newcommand{\disp}{\operatorname{disp}}
\newcommand{\len}{\operatorname{len}}
\newcommand{\mlen}{\operatorname{mlen}}

\newcommand{\w}{\xi}
\newcommand{\xw}{\underline{x}^{\xi}}
\newcommand{\xwp}{\underline{x}^{\xi'}}
\newcommand{\xwpp}{\underline{x}^{\xi''}}

\newtheorem{theorem}{Theorem}[section]

\newtheorem{proposition}[theorem]{Proposition}

\newtheorem{corollary}[theorem]{Corollary}

\newtheorem*{mainthm}{Main Theorem}

\newtheorem*{prop:levelcase}{Proposition \ref{levelcase}}

\theoremstyle{definition}

\newtheorem{definition}[theorem]{Definition}

\newtheorem{example}[theorem]{Example}

\theoremstyle{remark}

\newtheorem{remark}[theorem]{Remark}

\newtheorem{question}[theorem]{Question}

\journal{Journal of Pure and Applied Algebra}

\begin{document}

\begin{frontmatter}

\title{The Frobenius complexity of Hibi rings \tnoteref{t1}}
\tnotetext[t1]{\copyright 2018. This manuscript version is made available under the CC-BY-NC-ND 4.0 license \url{http://creativecommons.org/licenses/by-nc-nd/4.0/}.}

\author[label1]{Janet Page}
\address[label1]{Department of Mathematics, Statistics, and Computer Science, University of Illinois at Chicago, Chicago, IL 60607, USA}
\ead{jpage8@uic.edu}

\begin{abstract}
We study the Frobenius complexity of Hibi rings over fields of characteristic $p > 0$.  In particular, for a certain class of Hibi rings (which we call $\omega^{(-1)}$-level), we compute the limit of the Frobenius complexity as $p \rightarrow \infty$.
\end{abstract}


\end{frontmatter}

\section{Introduction}
Central to the study of singularities in characteristic $p$ is the Frobenius morphism and its splittings.  Given a commutative ring $R$ of positive characteristic, the total Cartier algebra ($\mathcal{C}(R)$) is the graded, noncommutative ring of all potential Frobenius splittings of $R$, and it has been studied in various contexts in its relation to singularities \cite{LSTestIdeal}, \cite{STestIdealsinNonQGorRings}, \cite{BTestIdealsofpeLinearMaps}.  Unfortunately, this ring is not finitely generated over $R$, even for relatively nice rings \cite{MBZFrobandCartAlgofSRRings}, \cite{KSSZRingsofFrobOp}, \cite{KAnExofNonfingenalgofFrobMaps}, but we can study the degree to which it is non-finitely generated.  Specifically, Enescu and Yao defined a measure of the non-finite generation of this ring \cite{EYTheFrobcompofalocalringofprimechar}, called the Frobenius complexity ($cx_F(R)$), and they computed it for Segre products of polynomial rings.  No other examples have been computed, and it is difficult to compute in a fixed characteristic.  In particular, Enescu and Yao found that the Frobenius complexity of a Segre product is not even always rational.  However, when they varied the base field and took a limit as $p \rightarrow \infty$,  they found the limit Frobenius complexity is an integer in every case they studied. We will focus on computing the limit Frobenius complexity for a class of toric rings called Hibi rings, which are defined using finite posets.  We will be able to compute it for Hibi rings which have a property we call $\omega^{(-1)}$-level, a condition related to the level condition which has been studied for Hibi rings in \cite{MASuffCondforaHibiRingtobeLevel} and \cite{MOnGenCanModuleHibiRing}.
Our main theorem shows that the limit Frobenius complexity is an integer for $\omega^{(-1)}$-level Hibi rings, and in fact it can be read off directly from the poset.  Specifically, we have the following.

\begin{mainthm}[Theorem \ref{cxflevel}]
If $R = \mathcal{R}_{\mathbb{F}_p}[\mathcal{I}(P)]$ is an $\omega_R^{(-1)}$-level (but non Gorenstein) Hibi ring associated to a poset $P$ over $\mathbb{F}_p$, then
\begin{equation*}
    \lim_{p \rightarrow \infty} cx_F(R) = \# \{\text{elements of P which are not in a maximal chain of minimal length}\}.
\end{equation*}
\end{mainthm}

Otherwise, in the Gorenstein case, we know $\mathcal{C}(R)$ is finitely generated over $R$, which means $cx_F(R) = -\infty$ \cite{LSTestIdeal}.  As a particular case of this theorem, we recover the result of Enescu and Yao on the limit Frobenius complexity of Segre products of polynomial rings.

Frobenius complexity quantifies the minimal number of generators of $\mathcal{C}(R)_e$ for any $e$, which cannot be written as products of elements of lower degrees.  We will give an upper bound  on the number of generators of $\mathcal{C}(R)_e$ using the toric structure of Hibi rings.  Then, we will use base $p$ expansion techniques to give a lower bound by explicitly finding generators which are not products of elements of lower degrees.  We will show these have the same order when our Hibi ring is $\omega^{(-1)}$-level.

\textbf{Acknowledgments:} This work was partially supported by NSF RTG grant DMS-1246844.  I would like to thank my advisor, Kevin Tucker, for his constant support and guidance, and Alberto Boix, J\"urgen Herzog, Chelsea Walton, Wenliang Zhang, and Dumitru Stamate for useful conversations and feedback.

\section{Background}
\subsection{Frobenius complexity}
Let $R$ be a ring of characteristic $p > 0$.  Then for any $R$-module $M$, we can consider the set of $p^{-e}$-linear maps on $M$, namely all maps $\psi: M \rightarrow M$ such that 
\begin{align*}
\psi(r^{p^e}m) &= r\psi(m) \text{ and } \\
\psi(m_1 + m_2) &= \psi(m_1) + \psi(m_2)   
\end{align*}
which we will denote $\mathcal{C}^e(M)$.

Similarly, we could consider the set of $p^e$-linear maps on $M$ which we denote $\mathscr{F}^e(M)$.  Namely, these are the maps $\phi: M \rightarrow M$ such that
\begin{align*}
\phi(rm) &= r^{p^e}\phi(m) \text{ and} \\
\phi(m_1 + m_2) &= \phi(m_1) + \phi(m_2)    
\end{align*}

Let $F^e: R \rightarrow R$ be the iterated Frobenius map, and let $F_*^eR$ denote the $R$-module which is isomorphic to $R$ as a set (we write elements in $F_*^eR$ as $F_*^er$ for some $r \in R$), but with an $R$-module structure given by:
\begin{equation*}
    r \cdot F_*^{e}x := F_*^e( r^{p^e}x) \text{ for all } r,x \in R
\end{equation*}
Similarly, for an $R$-module $M$, we let $F_*^eM$ be the $R$ module which agrees with $M$ as a set and has the multiplication structure $r \cdot F_*^e m = F_*^e(r^{p^e}m)$.

We can identify:
\begin{equation*}
    \mathcal{C}^e(M) \cong \Hom_R(F_*^eM,M)
\end{equation*}
and similarly 
\begin{equation*}
    \mathscr{F}^e(M) \cong \Hom_R(M, F_*^eM)
\end{equation*}

\begin{definition}(\cite{STestIdealsinNonQGorRings},\cite{LSTestIdeal}) Let $\mathcal{C}(M) = \oplus_{e} \mathcal{C}^e(M)$.  We call this the Cartier algebra on $M$.
Similarly, let $\mathscr{F}(M) = \oplus_{e} \mathscr{F}^e(M)$, which we call ring of Frobenius operators on $M$.
\end{definition}

We note that if $\phi_1 \in \mathscr{F}^{e_1}(M)$ and $\phi_2 \in \mathscr{F}^{e_2}(M)$ then $\phi_2 \circ \phi_1 (rm) = \phi_2(r^{p^{e_1}}\phi_1(m)) = r^{p^{e_1 + e_2}} \phi_2 \circ \phi_1 (m)$ so that $\phi_2 \circ \phi_1 \in \mathscr{F}^{e_1 + e_2}(M)$, and $\mathscr{F}(M)$ forms a graded ring (and likewise for $\mathcal{C}(M)$).

When $R$ is a complete, local, and $F$-finite ring, and $E = E_R(k)$, we have \cite{MBZFrobandCartAlgofSRRings}:
\begin{equation*}
\mathscr{F}(E)^{op} \cong \mathcal{C}(R)
\end{equation*}
We will define Frobenius complexity as a measure of the non-finite generation of $\mathcal{C}(R)$; however, when $R$ is complete and local this is equivalent to defining the same notion for $\mathscr{F}(E)$.

For any $\bbN$-graded ring $A = \oplus A_e$, let $G_e(A)$ be the subring of A generated by the elements of degree $\leq e$ and let $G_{-1} = A_0$.

\begin{definition}Let $c_e = c_e(A)$ denote the minimal number of homogeneous generators of $A_e/(G_{e-1}(A))_e$ over $A_0$.  We say $A$ is degree-wise finitely generated if $c_e < \infty$ for all $e$.
When $A$ is a degree-wise finitely generated, the sequence $\{c_e\}_e$ is called the \textbf{complexity sequence} for $A$.
\end{definition}

\begin{definition}\cite{EYTheFrobcompofalocalringofprimechar}
If $R$ is an $F$-finite ring of characteristic $p > 0$, we say the \textbf{Frobenius complexity} of $R$ is
$$
cx_F(R) = \inf\{ \alpha \in \bbR_{> 0} : c_e(\mathcal{C}(R)) = \mathcal{O}(p^{\alpha e})\}.
$$
\end{definition}

Throughout this paper, $R$ will be a normal ring, and we will denote its canonical module $\omega_R$.  When there is no confusion, we will write $\omega$ in place of $\omega_R$.  Let $\omega^{(-n)} := \Hom_R(\omega^{(n)}, R)$.  To compute Frobenius complexity, we will be using the following correspondence \cite{MRFrobeniussplittingandcohovanishing}, \cite{SZFrobspilttingincomalg}.  We state a version similar to Theorem 3.3 in \cite{KSSZRingsofFrobOp}, but note that we can state it more generally for $\mathcal{C}(R)$ which is isomorphic to $\mathscr{F}(E)^{op}$ when $R$ is complete and local.

\begin{theorem}\label{structureofCR}
Let $R$ be a normal, local ring of characteristic $p > 0$ with canonical module $\omega$. Then there is an isomorphism of graded rings
    \begin{equation*}
        \mathcal{C}(R) \cong T(\oplus_{n \geq 0} \omega^{(-n)})
    \end{equation*}
where if $A = \oplus_n A_n$ is a graded ring of characteristic $p > 0$, then  $T(A)$ is the graded ring with $e^{th}$ graded piece $T(A)_e = A_{p^{e} -1}$ with the following multiplication structure:
\begin{equation*}
    a * b = a^{p^{e'}}b \text{ for } a \in T(A)_e, b \in T(A)_{e'}
\end{equation*}
\end{theorem}

We will first give an upper bound on the limit Frobenius complexity.  In order to do this, we define a notion similar to analytic spread for the anticanonical, where we take the symbolic Rees algebra instead of the standard Rees algebra.

\begin{definition}
Let $(R,\mathfrak{m},k)$ be a normal, local ring with canonical $\omega$, and suppose $\oplus_{n \geq 0} \omega^{(-n)}$ is finitely generated.  
We define the \textbf{anticanonical symbolic spread} of $R$ to be:
\begin{equation*}
    sp_R(\omega^{(-1)}):= \dim\left(\bigoplus_{n \geq 0} \omega^{(-n)}\otimes k \right)
\end{equation*}
\end{definition}

We will show in Proposition \ref{wngenbyw} that $\bigoplus_{n \geq 0} \omega^{(-n)}$ is generated by $\omega^{(-1)}$ and the anticanonical symbolic spread matches up with the usual analytic spread for Hibi rings.  We will use this anticanonical symbolic spread to bound the limit Frobenius complexity in the following way.

\begin{proposition}\label{anticanonicalbound}
If $R$ is a normal, local ring over a field of characteristic $p$, and $\oplus_{n \geq 0} \omega^{(-n)}$ is finitely generated, then
\begin{equation*}
    \lim_{p \rightarrow \infty} cx_F(R) \leq sp_R(\omega^{(-1)}) -1
\end{equation*}
\end{proposition}
\begin{proof}
Consider the function $h :\mathbb{N} \rightarrow \mathbb{N}$ such that $h(n)$ is the minimal number of generators of $\omega^{(-n)}$.  When $\oplus_{n \geq 0} \omega^{(-n)}$ is generated by $\omega^{(-1)}$, then for $n >> 0$, $h$ matches up with some polynomial of degree $sp_R(\omega^{(-1)}) -1$.  Otherwise, for $n >> 0$, $h(n)$ will match up with some quasi-polynomial, so that there are some polynomials $\{p_i\}$ and for each sufficiently large $n$, $h(n) = p_i(n)$ for some $p_i$.  Again, the degrees of the $p_i$ are bounded by $sp_R(\omega^{(-1)}) -1$.  Then since $c_e(\mathcal{C}(R)) \leq h(p^e - 1)$ by Theorem \ref{structureofCR}, we have that $\lim_{p \rightarrow \infty} cx_F(R) \leq sp_R(\omega^{(-1)}) -1$.
\end{proof}

\subsection{Hibi rings}

This paper will focus on studying the limit Frobenius complexity of Hibi rings, which are normal toric rings which can be defined by finite posets.  These rings are particularly nice to work with as they can be described combinatorially.

\begin{definition}
Consider a poset $P$.  We call $I$ a \textbf{poset ideal} of $P$ if $y \in I \implies x \in I$ for all $x \leq y$, $x,y \in P$.  Denote by $\mathcal{I}(P)$ the set of all poset ideals in $P$.
\end{definition}

Note that $\mathcal{I}(P)$ forms a distributive lattice with a natural ordering.  For $I,J \in \mathcal{I}(P)$, we say $I \leq J$ if $I \subset J$ as subsets of $P$.  If $I,J \in \mathcal{I}(P)$ are incomparable ($I \not\leq J$ and $J \not\leq I$), we write $I \not\sim J$.  Note that $\mathcal{I}(P)$ always has a unique minimal element, namely $\varnothing$.

By Birkhoff's Theorem, we know that for any finite distributive lattice $D$, we have $D \cong \mathcal{I}(P)$ for some poset $P$.

\begin{definition} \cite{HDisplat} Given a poset $P:= \{v_1,\dots,v_n\}$ and a field $k$, the \textbf{Hibi ring} associated to $P$ over $k$, denoted $\mathcal{R}_k[\mathcal{I}(P)] \subset k[t,x_{v_1},\dots,x_{v_n}]$, is the toric ring generated over $k$ by the monomials $x_I := t\prod_{v_i \in I} x_{v_i}$ for every $I \in \mathcal{I}(P)$. 
\end{definition}

Note that $\mathcal{R}_k[\mathcal{I}(P)]$ has a natural grading given by the degree of $t$.

\begin{theorem}\cite{HDisplat}\label{hibiiso}
For any Hibi ring, we have the following isomorphism:
$$\mathcal{R}_k[\mathcal{I}(P)] \cong \frac{k[X_\alpha | \alpha \in \mathcal{I}(P)]}{(X_\alpha X_\beta - X_{\alpha \wedge \beta}X_{\alpha \vee \beta} | \alpha \not\sim \beta)}$$
where $\alpha \vee \beta, \alpha \wedge \beta$ are the join and meet of $\alpha$ and  $\beta$ (least upper bound and greatest lower bound respectively). 
\end{theorem}

\begin{definition}
For a poset $P$, we say $v_j$ \textbf{covers} $v_i$ when $v_i < v_j$ and there is no element $v_k$ such that $v_i < v_k < v_j$.  We denote this by $v_i \lessdot v_j$.
\end{definition}

We call a totally ordered subset of $P$ a chain, and we say $P$ is pure if all maximal chains have the same number of elements.  Hibi showed the following.

\begin{theorem} \cite{HDisplat} \label{pure}
A Hibi ring $R = \mathcal{R}_k[\mathcal{I}(P)]$ is Gorenstein if and only if $P$ is pure.
\end{theorem}

We draw each poset $P$ as a Hasse diagram.  For each element in $P$, we draw a vertex.  We connect two vertices $v_i$ and $v_j$ by an edge iff $v_i \lessdot v_j$, and in this case we draw $v_i$ below $v_j$.  In particular, if $v_i \leq v_j$, there is some path connecting $v_i$ and $v_j$.

\begin{example}\label{segre}

 In general, the Segre product of $k[x_1,\dots, x_n]$ and $k[y_1,\dots, y_m]$, denoted $S_{m,n}$, is equal to $\mathcal{R}_k[\mathcal{I}(P)]$, where $P$ is the poset given by a chain of length $m-1$ and a chain of length $n-1$.  We will explicitly see this example when $m = 2$ and $n = 1$.
    
 Consider the poset $P = \{v_1,v_2,v_3\}$ with one relation: $v_1 > v_2$.

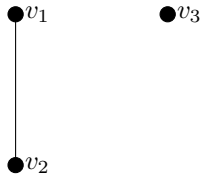
\vspace{.25in}
\begin{center}
\begin{tikzpicture}
\draw (0,0) -- (0,2);
\draw[fill] (0,0) circle [radius=0.1];
\draw[fill] (0,2) circle [radius=0.1];
\draw[fill] (2,2) circle [radius=0.1];
\node [right] at (0,2) {$v_1$};
\node [right] at (0,0) {$v_2$};
\node [right] at (2,2) {$v_3$};
\end{tikzpicture}
\end{center}

We have that $\mathcal{I}(P) =\{ \varnothing, \{v_3\}, \{v_2\}, \{v_1,v_2\}, \{v_2,v_3\}, \{v_1,v_2,v_3\}\}$ and the associated Hibi ring is
\begin{align*}
    k[t,tx_{v_3},tx_{v_2},tx_{v_1}x_{v_2},tx_{v_2}x_{v_3},tx_{v_1}x_{v_2}x_{v_3}] &\cong \frac{k[a,b,c,d,e,f]}{(ae-bd,af-cd,bf-ce)}\\ 
    &\cong k[x,y,z] \sharp k[u,v] := S_{3,2}
\end{align*}
where we use $\sharp$ to mean Segre product: if $A$ and $B$ are both $\mathbb{N}$-graded rings with $A_0 = R = B_0$ then $A \sharp B := \oplus_n (A_n \otimes_R B_n)$.  The first isomorphism is Theorem \ref{hibiiso}, and can be seen explicitly by setting $a = t, b = tx_{v_2}, c = tx_{v_1}x_{v_2}, d= tx_{v_3}, e  = tx_{v_2}x_{v_3}, f = tx_{v_1}x_{v_2}x_{v_3}$ and the second is the standard presentation of a Segre product as a determinantal ring.

Note that by Theorem \ref{pure}, $S_{m,n}$ is Gorenstein if and only if $n = m$.  When $m \neq n$ we will recover the result of Enescu and Yao, which states that for the Segre product $S_{m,n}$ with $m > n \geq 2$, we have $\lim_{p \rightarrow \infty} cx_F(S_{m,n}) = m -1$ \cite{EYOntheFrobcompofdetrings}.  
\end{example}

We associate to a poset $P$ another poset $\hat{P}:= \{P,-\infty,\infty\}$ which is $P$ with a minimal element $-\infty$ and a maximal element $\infty$ added.

\begin{definition}
We say a tuple $p = (v_1, \dots ,v_k)$ of distinct elements of $P$ is a \textbf{path} if for each adjacent $v_i,v_{i+1} \in p$ we have $v_i \lessdot v_{i+1}$ or $v_{i+1} \lessdot v_i$.  We do not allow paths to repeat elements.  We say $p$ is an \textbf{upwards path} if for all $1 \leq i \leq k-1$, we have $v_i \lessdot v_{i+1}$.  Similarly, we say $p$ is a \textbf{downwards path} if for all $1 \leq i \leq k-1$ we have $v_{i+1} \lessdot v_i$.  We call $p$ a \textbf{mixed path} if $p$ is a path but is not an upwards or downwards path.  
If $p = (v_1,\dots,v_k)$ is an upwards or a downwards path (i.e. if $p$ is a chain), we say its \textbf{length} is $k-1$, which we denote $\len(p)$.
\end{definition}

The following definitions will help us simplify the statement of our main theorem.
\begin{definition}
We call the set of all vertices which lie on a minimal length upwards path from $-\infty$ to $\infty$, the \textbf{minimal subset} of $\hat{P}$, which we denote $\hat{P}_{min}$, i.e. 
\begin{equation*}
    \hat{P}_{min} =\{v \in \hat{P} | v \in p \text{ for some minimal length upwards path } p \text{ from } -\infty \text{ to } \infty\}
\end{equation*}
Similarly, we define its compliment:
\begin{equation*}
    \hat{P}_{nonmin} =\{v \in \hat{P} | v \notin p \text{ for any minimal length upwards path } p \text{ from } -\infty \text{ to } \infty \}
\end{equation*}
which we call the \textbf{nonminimal subset} of $\hat{P}$.
\end{definition}

We will show that if a Hibi ring $R = \mathcal{R}_{\mathbb{F}_p}[\mathcal{I}(P)]$ is $\omega^{(-1)}$-level (but not Gorenstein), then
$$
\lim_{p \rightarrow \infty} cx_F(R) = |\hat{P}_{nonmin}|.
$$

We will use the following description of $R$.

\begin{proposition}\label{definingequations}
Let $R = \mathcal{R}_k[\mathcal{I}(P)]$ and let $\xi$ be a map $\xi: \hat{P} - \{\infty\} \rightarrow \mathbb{Z}$.  Let $\xw = t^{\xi(-\infty)}\prod_{v\in P} x_v^{\xi(v)}$.  Then $\xw \in R$ if and only if for all $v_i,v_j \in \hat{P} - \{\infty\}$ we have:
\begin{align*}
    v_i &\lessdot v_j \implies \xi(v_i) - \xi(v_j) \geq 0 \text{ and } \\
    v_i &\lessdot \infty \; \implies \xi(v_i) \geq 0
\end{align*}
\end{proposition}
\begin{proof}
$(\leftarrow)$ Assume $\w$ is as above, so that $\xi(v_i) \geq \xi(v_j)$ for all $v_i \lessdot v_j$ with $v_i,v_j \in \hat{P} - \{\infty\}$ and $\xi(v_i) \geq 0$ for all $v_i \lessdot \infty$ (i.e. for all $v_i$ maximal in $P$).  We want to show $\xw \in R$.  First note that these two conditions force $\xi(v_i) \geq 0$ for all $v_i \in \hat{P} - \{\infty\}$.  We induct on $\sum \xi(v_i)$.  If $\sum \xi(v_i) = 0$, then $\xi(v_i) = 0$ for all $v_i$, so $\xw = 1 \in R$.  Now suppose $\sum \xi(v_i) > 0$ for some $\xw$ and suppose that $\xwp \in R$ as long as $\sum \xi'(v_i) < \sum \xi(v_i)$.  Let $I = \{ v_i \in P | \xi(v_i) = \xi(-\infty)\}$.  We will show that $I$ is a poset ideal of $P$ so that $x_I \in R$ by definition, and $\xw = x_I\xwp$ for some $\xwp$ with $\sum \xi'(v_i) < \sum \xi(v_i)$ so that $\xwp \in R$ and therefore $\xw \in R$.  Suppose $v_i \in I$ and suppose there is some $v_l \leq v_i$.  We know $-\infty$ is minimal so $-\infty \leq v_l \leq v_i$ and we can find an upwards path $p = (-\infty, \dots, v_l , \dots, v_i)$ so that:
\begin{equation*}
    -\infty \lessdot \dots \lessdot v_l \lessdot \dots  \lessdot v_i
\end{equation*}
Then we have a string of inequalities corresponding to $p$:
\begin{equation*}
    \xi(-\infty) \geq \dots \geq \xi(v_l) \geq \dots  \geq \xi(v_i)
\end{equation*}
In particular, since $\xi(-\infty) = \xi(v_i)$, these all must be equalities and so $v_l \in I$ as well.  Thus, $I$ defines a poset ideal so that $x_I \in R$, and by construction $x_I | \xw$ so that $\xw = x_I\xwp$  for some $\xwp$ with $\sum \xi'(v_i) < \sum \xi(v_i)$, and thus $\xw \in R$.
\\
$(\rightarrow)$ Now suppose $\xw \in R$.  Clearly $\xi(v_i) \geq 0$ for all $i$ and so in particular $\xi(v_i) \geq 0$ for all $v_i \lessdot \infty$ in $\hat{P}$.  Note that $\xw = x_{I_1} \dots x_{I_n}$ where each monomial $x_{I_l}$ corresponds to a poset ideal $I_l$.  Let $\xi(v_j) > 0$ in $\w$ and say $v_i \lessdot v_j$.  In particular, $x_{v_j}$ must show up in $\xi(v_j)$ of the monomials, since each of these has at most one $x_{v_j}$ by construction. Then $v_j$ is in the $\xi(v_j)$ corresponding poset ideals and since $v_i \leq v_j$, we must have that $v_i$ shows up in these $\xi(v_j)$ poset ideals as well and thus $x_{v_i}$ shows up in each of the $\xi(v_j)$ corresponding monomials and so $\xi(v_i) \geq \xi(v_j)$.    
\end{proof}

We can extend $\xi$ to a map $\xi: \hat{P} \rightarrow \mathbb{Z}$ by letting $\xi(\infty) = 0$.  Given such a map, we will continue to use the following notation: $\xw = t^{\xi(-\infty)}\prod_{v\in P} x_v^{\xi(v)}$. Then we can simplify Proposition \ref{definingequations} to the following.

\begin{corollary}\label{definingeqsimplified}
Let $R = \mathcal{R}_k[\mathcal{I}(P)]$, and let $\xi$ be a map $\xi: \hat{P} \rightarrow \mathbb{Z}$ with $\xi(\infty) = 0$.  Then $\xw \in R$ if and only if for all $v_i,v_j \in \hat{P}$ we have:
\begin{equation*}
    v_i \lessdot v_j \implies \xi(v_i) - \xi(v_j) \geq 0
\end{equation*}
\end{corollary}

Finally, we introduce two definitions which will help us count generators of $\omega^{(-n)}$:

\begin{definition}
For any $v_i \leq v_j \in \hat{P}$, we define the \textbf{distance} between $v_i$ and $v_j$, denoted $\dist(v_i,v_j)$ to be the length of the shortest upwards path between $v_i$ and $v_j$.
\end{definition}

Note that we have the following triangle-like inequality, for any $v_i \leq v_j \leq v_k$
\begin{equation}\label{triangle}
    \dist(v_i,v_j) + \dist(v_j,v_k) \geq \dist(v_i,v_k)
\end{equation}

\begin{definition}
For any $v_i \leq v_j \in \hat{P}$ we define the \textbf{disparity between \textit{v}$_i$ and \textit{v}$_j$} to be the difference in length between a minimal length upwards path from $v_i$ to $\infty$, and a minimal length upwards path between $v_i$ and $\infty$ through $v_j$, namely:
\begin{equation*}
    \disp(v_i,v_j) := \dist(v_i,v_j) + \dist(v_j,\infty) - \dist(v_i,\infty)
\end{equation*}
\end{definition}

Note that by (\ref{triangle}), we always have $\disp(v_i,v_j) \geq 0$ for any $v_i \leq v_j \in \hat{P}$.

\section{\texorpdfstring{Structure of $\oplus_{n \geq 0} \omega^{(-n)}$}{Structure of the anticanonical symbolic Reese algebra}}
In this section, we study $\oplus_n \omega^{(-n)}$ in order to bound the limit Frobenius complexity using Proposition \ref{anticanonicalbound}.  First, we explicitly find generators of $\omega^{(-n)}$ for each $n$.

\begin{corollary}\label{generatorsforcan}
Let $R = \mathcal{R}_k[\mathcal{I}(P)]$, and let $\xi: \hat{P} \rightarrow \mathbb{Z}$, with $\xi(\infty) = 0$.  As before, let $\xw =  t^{\xi(-\infty)}\prod_{v\in P} x_v^{\xi(v)}$.  Then $\xw \in \omega^{(-n)}$ if and only if $\xi$ satisfies the following:
\begin{align}
\label{secondeq}
\xi(v_i) &\geq \xi(v_j) - n \text{ for } v_i \lessdot v_j \in \hat{P}
\end{align}
\end{corollary}
\begin{proof}
This follows from Corollary \ref{definingeqsimplified} by considering $R$ as a toric ring.  See for example \cite{FIntrotoToricVarieties}, Section 4.3 and \cite{bruns2009polytopes}, 4.F.
\end{proof}

In particular, we have that
\begin{equation} \label{lowerboundwi}
    \xi(v) \geq -\dist(v,\infty)n \text{ for all } v \in \hat{P}.
\end{equation}

When $R$ is a Hibi ring with canonical $\omega$, we will show that $\oplus_n \omega^{(-n)}$ is generated by $\omega^{(-1)}$, and so Proposition \ref{anticanonicalbound} holds.  This also tells us that $\oplus_{n \geq 0} \omega^{(-n)} = \oplus_{n \geq 0} \omega^{-n}$, and $sp_R(\omega^{(-1)})$ is just the usual analytic spread of the anticanonical.

We use the following notation: For any $a,m \in \bbZ$, let $[a]_{m}$ be the representative of $a \pmod{m}$ between $0$ and $m -1$.

\begin{proposition}\label{wngenbyw}
Suppose we have a minimal generator $\xw \in \omega^{(-n)}$ over $R$.  Then we can find $\underline{x}^{\xi_1} \cdots \underline{x}^{\xi_n} = \xw$ such that for each $l$, $\underline{x}^{\xi_l}$ is a minimal generator of $\omega^{(-1)}$ over $R$.  In particular, $\oplus_n \omega^{(-n)}$ is finitely generated, and it is generated by $\omega^{(-1)}$.
\end{proposition}
\begin{proof}
A minimal generator $\xw$ of $\omega^{(-n)}$, gives us $\xi: \hat{P} \rightarrow \bbZ$ satisfying $\xi(\infty) = 0$ and (\ref{secondeq}).

It suffices to find partitions $\xi(v_i) = \xi_1(v_i) + ... + \xi_n(v_i)$ such that for all $l$
\begin{align*}
&\xi_l(\infty) = 0 \text{ and } \\
&\xi_l(v_i) \geq \xi_l(v_j) - 1 \text{ for } v_i \lessdot v_j
\end{align*}

For each $i$, if $n|\xi(v_i)$ then let $\xi_l(v_i) = \frac{\xi(v_i)}{n}$ for all $0 \leq l \leq n$.  Otherwise, let 
\begin{align*}
\xi_l(v_i) &= \left\lceil \frac{\xi(v_i)}{n} \right\rceil \text{ for } 1 \leq l \leq [\xi(v_i)]_{n} \text{ and } \\
\xi_l(v_i) &= \left\lfloor \frac{\xi(v_i)}{n} \right\rfloor \text{ for }
[\xi(v_i)]_{n} < l \leq n.
\end{align*}

In particular, we have
\begin{equation*}
   \sum_{l=1}^n \xi_l(v_i) = \xi(v_i).
\end{equation*}

Consider $v_i \lessdot v_j$ and $1 \leq l \leq n$.  There are four cases.

Case 1,2,3: 
\begin{align*}
   \xi_l(v_i) &= \left\lfloor \frac{\xi(v_i)}{n} \right\rfloor, \; \xi_l(v_j)= \left\lfloor \frac{\xi(v_j)}{n} \right\rfloor \\
   \xi_l(v_i) &= \left\lceil \frac{\xi(v_i)}{n} \right\rceil, \; \xi_l(v_j) = \left\lfloor \frac{\xi(v_j)}{n} \right\rfloor \\
   \xi_l(v_i) &= \left\lceil \frac{\xi(v_i)}{n} \right\rceil, \; \xi_l(v_j) = \left\lceil \frac{\xi(v_j)}{n} \right\rceil 
\end{align*}
In all of these cases, we have
\begin{equation*}
    \xi_l(v_i) \geq \xi_l(v_j) - 1
\end{equation*}
since
\begin{equation*}
    \xi(v_i) \geq \xi(v_j) - n.
\end{equation*}
Case 4:
\begin{equation*}
    \xi_l(v_i) = \left\lfloor \frac{\xi(v_i)}{n} \right\rfloor, \; \xi_l(v_j)= \left\lceil \frac{\xi(v_j)}{n} \right\rceil
\end{equation*}
For this to happen, we must have that $[\xi(v_i)]_n < [\xi(v_j)]_n$.
We know:
\begin{align*}
    \xi(v_i) &\geq \xi(v_j) - n \\
    \implies \left\lfloor \frac{\xi(v_i)}{n} \right\rfloor + [\xi(v_i)]_n &\geq \left\lfloor \frac{\xi(v_j) - n}{n}  \right\rfloor + [\xi(v_j) -n]_n \\
    \implies\left\lfloor \frac{\xi(v_i)}{n} \right\rfloor + [\xi(v_i)]_n &\geq \left\lfloor \frac{\xi(v_j)}{n}  \right\rfloor - 1 + [\xi(v_j)]_n.
\end{align*}
Thus, since $[\xi(v_i)]_n < [\xi(v_j)]_n$, we must have:
\begin{equation*}
     \left\lfloor \frac{\xi(v_i)}{n} \right\rfloor > \left\lfloor \frac{\xi(v_j)}{n}  \right\rfloor - 1
\end{equation*}
so that
\begin{equation*}
     \left\lfloor \frac{\xi(v_i)}{n} \right\rfloor \geq \left\lceil \frac{\xi(v_j)}{n}  \right\rceil - 1
\end{equation*}
as desired.
\end{proof}

Because $\oplus_n \omega^{(-n)}$ is generated by $\omega^{(-1)}$, it will suffice to understand $\omega^{(-1)}$.  Recall that the grading of a Hibi ring is given by $t$, and we can naturally extend this grading to $\omega^{(-1)}$.  Specifically, if $\xw \in \omega^{(-1)}$, we say the degree of $\xw$ is $\xi(-\infty)$.  Our study will naturally break up into two cases: the case when all minimal generators of $\omega^{(-1)}$ have the same degree, and the case when $\omega^{(-1)}$ has minimal generators of varying degrees.

In \cite{Scmcomplexes}, Stanley introduced a notion of levelness for Cohen Macaulay standard graded $k$ algebras.  Specifically, his definition is equivalent to the following: $R$ is \textbf{level} if $\omega_R$ is generated in one degree.  Note that if $R$ is Gorenstein, $\omega$ is generated by one element and so in particular it is generated in one degree.  Thus, we have
\begin{equation*}
    \text{ Gorenstein rings } \subset \text{ level rings } \subset \text{ Cohen Macaulay rings.}
\end{equation*}

A classification of level Hibi rings is given in \cite{MOnGenCanModuleHibiRing}.

\begin{definition}
Following the definition above, we call a Hibi ring $\omega^{(-1)}$-\textbf{level} (anticanonical level) if $\omega^{(-1)}$ is generated in one degree over $R$.
\end{definition}

If $R$ is Gorenstein, we also have that $\omega^{(-1)}$ is generated by one element, so $R$ is also $\omega^{(-1)}$ level.  Again, we have
\begin{equation*}
    \text{ Gorenstein rings } \subset \omega^{(-1)}\text{-level rings } \subset \text{ Cohen Macaulay rings.}
\end{equation*}

However, we can find examples which are level but not $\omega^{(-1)}$-level, and likewise examples which are $\omega^{(-1)}$-level but not level.

\begin{example}\label{levelex}
\begin{enumerate}
\item Let $R = \mathcal{R}_k[\mathcal{I}(P)]$ be the Hibi ring associated to the poset $P$ below.
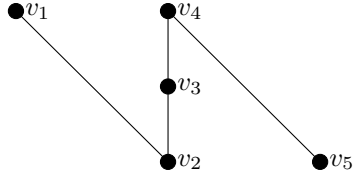
\begin{center}
\begin{tikzpicture}
\draw (-2,0) -- (0,-2) -- (0,0) -- (2,-2);
\draw[fill] (0,0) circle [radius=0.1];
\draw[fill] (0,-1) circle [radius=0.1];
\draw[fill] (0,-2) circle [radius=0.1];
\draw[fill] (-2,0) circle [radius=0.1];
\draw[fill] (2,-2) circle [radius=0.1];

\node [right] at (0,0) {$v_4$};
\node [right] at (0,-1) {$v_3$};
\node [right] at (0,-2) {$v_2$};
\node [right] at (-2,0) {$v_1$};
\node [right] at (2,-2) {$v_5$};
\end{tikzpicture}
\end{center}
Then one can check that $\omega^{(-1)}$ is generated by elements of degrees $-3$ and $-2$, and $R$ is not $\omega^{(-1)}$-level.  However, $\omega$ is generated by elements of degree $4$ and so $R$ is level.
\item Let $R = \mathcal{R}_k[\mathcal{I}(P)]$ be the Hibi ring associated to the poset $P$ below.
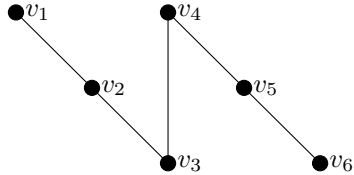
\begin{center}
\begin{tikzpicture}
\draw (-2,0) -- (0,-2) -- (0,0) -- (2,-2);
\draw[fill] (0,0) circle [radius=0.1];
\draw[fill] (-1,-1) circle [radius=0.1];
\draw[fill] (0,-2) circle [radius=0.1];
\draw[fill] (-2,0) circle [radius=0.1];
\draw[fill] (1,-1) circle [radius=0.1];
\draw[fill] (2,-2) circle [radius=0.1];

\node [right] at (0,0) {$v_4$};
\node [right] at (-1,-1) {$v_2$};
\node [right] at (1,-1) {$v_5$};
\node [right] at (0,-2) {$v_3$};
\node [right] at (-2,0) {$v_1$};
\node [right] at (2,-2) {$v_6$};
\end{tikzpicture}
\end{center}
Then $\omega^{(-1)}$ is generated by elements of degree $-3$.  However, $\omega$ generated by elements of degrees $4$ and $5$ and $R$ is not level.
\end{enumerate}
\end{example}

\begin{remark}\label{mingenform}
Note that if $R$ is $\omega^{(-1)}$- level, then by Proposition \ref{wngenbyw}, we also know $\omega^{(-n)}$ is generated in one degree (i.e., $R$ is also $\omega^{(-n)}$-level).  We can always find a minimal generator of $\omega^{(-n)}$ with $\xi(-\infty) = -\dist(-\infty,\infty)n$, by letting $\xi(v) = -\dist(v,\infty)n$ for all $v \in \hat{P}$. In particular, if $R$ is $\omega^{(-1)}$-level, then for every minimal generator $\xw$ of $\omega^{(-n)}$, we must have that $\xi(-\infty) = -\dist(-\infty,\infty)n$.
\end{remark}

\section{Frobenius complexity for anticanonical level Hibi rings}
In this section, we will compute the limit Frobenius complexity for $\omega^{(-1)}$-level Hibi rings.  In Theorem \ref{levelgen}, we characterize minimal generators $\xw$ of $\omega^{(-n)}$ by giving inequalities on the maps $\xi: \hat{P} \rightarrow \mathbb{Z}$ which give minimal generators.  Theorem \ref{orderofeq}  will give an upper bound on $cx_F(R)$ by computing $sp_R(\omega^{(-1)})$.  To get a lower bound, Proposition \ref{type1} shows that if $\xw = \xwpp*\xwp$ then $\xi',\xi''$ must satisfy the same type of inequalities as $\xi$ and we give a lower bound for these generators in  Theorem \ref{cxflevel}. The upshot is the following theorem.
\begin{theorem}[Theorem \ref{cxflevel}]
If $R = \mathcal{R}_{\mathbb{F}_p}[\mathcal{I}(P)]$ is not Gorenstein but is $\omega^{(-1)}$-level, then
\begin{equation*}
    \lim_{p \rightarrow \infty} cx_F(R) = sp_R(\omega^{(-1)}) -1 = |\hat{P}_{nonmin}|.
\end{equation*}
\end{theorem}

We start by writing inequalities on maps $\xi: \hat{P} \rightarrow \mathbb{Z}$ such that $\xi(\infty) = 0$ which determine all minimal generators $\xw$ of $\omega^{(-n)}$.

\begin{theorem}\label{levelgen}
If $R = \mathcal{R}_k[\mathcal{I}(P)]$ is $\omega^{(-1)}$-level and $\xi$ is a map $\xi: \hat{P} \rightarrow \mathbb{Z}$ such that $\xi(\infty) = 0$, then the following are equivalent:
\begin{enumerate}
    \item $\xw$ is a minimal generator of  $\omega^{(-n)}$, \\
    \item $\xi(-\infty) = -\dist(-\infty,\infty)n$ and $\xw \in \omega^{(-n)}$, and \\
    \item $\xi(v_i) = -\dist(v_i,\infty)n$ for all $v_i \in \hat{P}_{min}$  and \\ $-\dist(v_i,\infty)n \leq \xi(v_i) \leq \min\{\xi(v_k) + n |v_k \lessdot v_i\}$ for $v_i \in \hat{P}_{nonmin}$. \label{allwibounds}
\end{enumerate}
\end{theorem}
\begin{proof}
$(1) \implies (2)$
If $\xw$ is a minimal generator of $\omega^{(-n)}$, then clearly $\xw \in \omega^{(-n)}$, and by Remark \ref{mingenform}, $\xi(-\infty) = -\dist(-\infty,\infty)n$.
\\
$(2) \implies (3)$ Both of these inequalities come directly from Corollary \ref{generatorsforcan}
once we fix $\xi(-\infty) = -\dist(-\infty,\infty)n$ and $\xi(\infty) = 0$.
\\
$(3) \implies (1)$
The inequalities in (3) force $\xi$ to satisfy the inequalities in Corollary \ref{generatorsforcan}, so that $\xw \in \omega^{(-n)}$.  Furthermore, $\xw$ will be a minimal generator since $\xi(-\infty) = -\dist(-\infty,\infty)n$. To see this, note that for any $\xw \in \omega^{(-n)}$, we must have $\xi(-\infty) \geq -\dist(-\infty,\infty)n$.  In particular, if $\xi(-\infty) = -\dist(-\infty,\infty)n$ then the degree on $t$ is minimal and we cannot factor any element in $R$ out of $\xw$.
\end{proof}

To make counting easier, we define for each $\xw \in \omega^{(-n)}$ a map $N_{\w}: \hat{P} \rightarrow \bbZ$ by
\begin{equation*}
N_{\w}(v) = \xi(v) + \dist(v,\infty)n \text{ for all } v \in \hat{P}
\end{equation*}

We may write $N(v)$ in place of $N_{\w}(v)$  when $\w$ is implied.

\begin{corollary}\label{genmaps}
If $R$ is $\omega^{(-1)}$-level, then for every minimal generator $\xw$ of $\omega^{(-n)}$, we have a map $N: \hat{P} \rightarrow \bbZ$ such that
\begin{align}
N(v_i) &= 0 \;\; \text{ for all } v_i \in \hat{P}_{min} \text{ and}  \label{one} \\
0 \leq N(v_i) &\leq \min\{N(v_k) + \dis(v_k,v_i)n | v_k \lessdot v_i\} \text{ for } v_i \in \hat{P}_{nonmin} \label{two}
\end{align}
by letting $N(v_i) = \xi(v_i) + \dist(v_i,\infty)n$.  Furthermore, for every map $N$ satisfying the inequalities above, we can produce a minimal generator $\xw$ of $\omega^{(-n)}$ by letting $\xi(v_i) = N(v_i) - \dist(v_i,\infty)n$
\end{corollary}
\begin{proof}
Theorem \ref{levelgen} tells us that $\xw$ is a minimal generator if and only if $\xi(v_i) = -\dist(v_i,\infty)n$, i.e. if and only if $N(v_i) = 0$ for all $v_i \in \hat{P}_{min}$ and
\begin{align*}
    -\dist(v_i,\infty) &\leq \xi(v_i) \leq \min\{\xi(v_k) + n | v_k \lessdot v_i\} \text{ for } v_i \in \hat{P}_{nonmin} \\
    \Leftrightarrow 0 &\leq N(v_i) \leq \min\{\xi(v_k) + n + \dist(v_i,\infty)| v_k \lessdot v_i\} \text{ for } v_i \in \hat{P}_{nonmin} \\
    \Leftrightarrow 0 &\leq N(v_i) \leq \min\{N(v_k) -\dist(v_k,\infty) + n + \dist(v_i,\infty)| v_k \lessdot v_i\} \text{ for } v_i \in \hat{P}_{nonmin} \\
    \Leftrightarrow 0 &\leq N(v_i) \leq \min\{N(v_k) +\disp(v_k,v_i)n | v_k \lessdot v_i\} \text{ for } v_i \in \hat{P}_{nonmin}.
\end{align*}
\end{proof}

We will now count all maps $N: \hat{P} \rightarrow \bbZ$ which satisfy (\ref{one}) and (\ref{two}).  We will need the following definitions.

\begin{definition}
We call $v_i \in P$ a \textbf{top node} if it covers at least two distinct elements, i.e. if $v_k \lessdot v_i$ and  $v_j \lessdot v_i$ for some $v_j \neq v_k$.  Similarly, we call $v_i$ a \textbf{bottom node} if it is covered by at least two distinct elements, i.e. if $v_i \lessdot v_k$ and  $v_i \lessdot v_j$ for some $v_j \neq v_k$.
\end{definition}

\begin{definition}
We call $v_i \in P$ a \textbf{ starting point } if $v_i \in \hat{P}_{nonmin}$ and $v_i$ is a top node, or if $v_i \in \hat{P}_{nonmin}$ and $v_k \lessdot v_i$ for some $v_k \in \hat{P}_{min}$ or for some $v_k$ which is a bottom node.  Denote the set of starting points $\mathcal{S}$.  We call these starting points because they will correspond to the start of strings of inequalities. 
\end{definition}

\begin{theorem}\label{orderofeq}
If $R = \mathcal{R}_k[\mathcal{I}(P)]$ is $\omega^{(-1)}$-level, then for $n >> 0$ the number of generators of $\omega^{(-n)}$ over $R$ is $\Theta(n^{|\hat{P}_{nonmin}|})$.
In particular, this means $sp_R(\omega^{(-1)}) -1 = |\hat{P}_{nonmin}|$.
\end{theorem}
\begin{proof}
By Corollary \ref{genmaps}, it suffices to show that there are $\Theta(n^{|\hat{P}_{nonmin}|})$ maps $N: \hat{P} \rightarrow \bbZ$ that satisfy (\ref{one}) and (\ref{two}).

\textbf{Upper bound:} An upper bound for all such maps is those which instead satisfy
\begin{align*}
    N(v) &= 0 \text{ for } v \in \hat{P}_{min}, \text{ and } \\
    0 \leq N(v) &\leq M_vn \text{ for } v \in \hat{P}_{nonmin}
\end{align*}
for any $M_v > \sum_{v_j \lessdot v_l \leq v} \disp(v_j,v_l)n$.  For each $v \in \hat{P}_{nonmin}$, we have $M_vn = \Theta(n)$, options for $N(v)$.  Then overall, this gives $\Theta(n^{|\hat{P}_{nonmin}|})$ options for $N$ satisfying these inequalities.

\textbf{Lower bound:} 
For any $v \in \hat{P}_{nonmin}$, let $l_v$ be the number of starting points $\leq v$, and suppose $n > 2^{l_v + 1}$ for any $v$.
Then consider the following inequalities:
\begin{align}
    N(v) &= 0 \text{ for } v \in \hat{P}_{min}, \label{string1} \\
    \frac{n}{2^{l_v +1}} \leq N(v) &\leq \frac{n}{2^{l_v}} \text{ for } v \in \mathcal{S}, \text{ and} \label{string2} \\
    \frac{n}{2^{l_v+1}} \leq N(v) &\leq N(\tilde{v}) \text{ for } v \in \hat{P}_{nonmin}, v \notin \mathcal{S} \label{string3} 
\end{align}
where $\tilde{v}$ is the unique element such that  $\tilde{v} \lessdot v$.

We will show that if $N$ satisfies these inequalities, then it must also satisfy (\ref{one}) and (\ref{two}).  Clearly this holds for $v \in \hat{P}_{min}$.  Now consider $v \in \hat{P}_{nonmin}$.

Case 1: Suppose $v \in \mathcal{S}$, and consider any $w \lessdot v$.   Then $l_w +1 \leq l_v$.  
If $w \in \hat{P}_{nonmin}$, we have by definition $\frac{n}{2^{l_w+1}} \leq N(w)$.  Then
\begin{equation*}
    N(v) \leq \frac{n}{2^{l_v}} \leq \frac{n}{2^{l_w +1}} \leq N(w) \leq \disp(w,v)n + N(w)
\end{equation*}
as desired (and clearly $\frac{n}{2^{l_v+1}} \leq N(v) \implies 0 \leq N(v)$).  Now if $w \in \hat{P}_{min}$, we have $N(w) = 0$ and $\disp(w,v) > 0$ so that
\begin{equation*}
    N(v) \leq \frac{n}{2^{l_v}} \leq n \leq \disp(w,v)n = \disp(w,v)n + N(w).
\end{equation*}
Thus, for all $v \in \mathcal{S}$ we have:
\begin{equation*}
    0 \leq N(v) \leq \min\{\disp(w,v)n + N(w)| w \lessdot v\}.
\end{equation*}

Case 2: Now suppose $v \in \hat{P}_{nonmin}$, but $v \notin \mathcal{S}$.  Then there is a unique $\tilde{v}$ such that $\tilde{v} \lessdot v$, and we know $\tilde{v} \in \hat{P}_{nonmin}$.  Then we have
\begin{equation*}
    \frac{n}{2^{l_v+1}} \leq N(v) \leq N(\tilde{v})
\end{equation*}
and so for all $v \in \hat{P}_{nonmin}, v \notin S$, we have:
\begin{equation*}
    0 \leq N(v) \leq N(\tilde{v}) + \disp(\tilde{v},v)n = \min\{\disp(w,v)n + N(w)| w \lessdot v\} 
\end{equation*}
as desired.

Now we count all $N$ satisfying (\ref{string1})-(\ref{string3}).  Our set of starting point consists of top nodes, elements covering bottom nodes, and elements covering elements in $\hat{P}_{min}$.  Thus, we have broken $\hat{P}_{nonmin}$ into disjoint upwards paths which begin at starting points.  Each of these upwards paths is of the form  $v_1 \lessdot \dots \lessdot v_m$ where $v_1 \in \mathcal{S}$ and $v_m \lessdot y$ for some $y \in \mathcal{S}$ or $y \in \hat{P}_{min}$ and each $v \in \hat{P}_{nonmin}$ is in exactly one such path.  Each of these paths then gives us a string of inequalities of the form:
\begin{equation*}
    \frac{n}{2^{l_{v_1} + 1}} \leq N(v_m) \leq ... \leq N(v_1) \leq \frac{n}{2^{l_{v_1}}}
\end{equation*}

Then for each such string of $m$ elements, we get 
$$\binom{\frac{n}{2^{l_{v_1}+1}}}{m} = \Theta(n^{m})$$ 
options for $N(v_1),\dots,N(v_m)$, and so there are $\Theta(n^{|\hat{P}_{nonmin}|})$ choices of $\{N(v)\}_{v \in \hat{P}}$ which satisfy (\ref{string1})-(\ref{string3}).
Then there are at least $\Theta(n^{|\hat{P}_{nonmin}|})$ choices of $N$ satisfying (\ref{one}) and (\ref{two}).  
Along with our upper bound, this tells us that $sp_R(\omega^{(-1)}) -1 = |\hat{P}_{nonmin}|$.
\end{proof}

In order to compute Frobenius complexity, we will want to study which of the minimal generators of $\omega^{(1-p^e)}$ can be written as $a*b$ for $a$ and $b$ of smaller degrees.
We can show that if a minimal generator $\xw$ splits as $\xw = \xwpp*\xwp$, then $\xwp$ and $\xwpp$ are also minimal generators of the same type.

\begin{proposition}\label{type1}
If $R$ is a Hibi ring such that $\xw$ is a minimal generator of $\omega^{(1-p^{e})}$ with $\xi(-\infty) = -\dist(-\infty
,\infty)(p^e -1)$ and 

\begin{equation*}
    \xw = \xwpp * \xwp
\end{equation*}
for $\xwp \in \omega^{(1-p^{e'})}$ and $\xwpp \in \omega^{(1-p^{e''})}$, then $\xwp,\xwpp$ are minimal generators of $\omega^{(1-p^{e'})}$ and $\omega^{(1-p^{e''})}$ of the same type (i.e. with $\xi'(-\infty) = -\dist(-\infty
,\infty)(p^{e'} -1)$ and $\xi''(-\infty) = -\dist(-\infty
,\infty)(p^{e''} -1)$).
\end{proposition}
\begin{proof}
By assumption, $\xi(-\infty) = -\dist(-\infty,\infty)(\pe)$. If $\xw = \xwpp*\xwp$, then we must have:
\begin{equation*}
    \xi(-\infty) = \xi''(-\infty)p^{e'} + \xi'(-\infty)
\end{equation*}
so that since $\xi'(-\infty) \geq -\dist(-\infty,\infty)(p^{e'} -1)$ and $\xi''(-\infty) \geq -\dist(-\infty,\infty)(p^{e''} -1)$  we must have that these are actually equalities. Then $\xwp$ and $\xwpp$ are indeed minimal generators of $\omega^{(1-p^{e'})}$ and $\omega^{(1-p^{e''})}$ since their degrees on $t$ are minimal.
\end{proof}

We know if $R = \mathcal{R}_k[\mathcal{I}(P)]$ is a Gorenstein Hibi ring with canonical $\omega$ where $k$ is a field of characteristic $p$, then $sp_R(\omega^{(-1)})-1 = 0$ (since $\omega^{(-n)}$ is generated by $1$ element for every $n$), and we have seen $cx_F(R) = -\infty$ \cite{EYTheFrobcompofalocalringofprimechar}.  
Otherwise, we will show that when $R$ is $\omega^{(-1)}$-level, we have $\lim_{p \rightarrow \infty} cx_F(R) = sp_R(\omega^{(-1)}) -1$.  
We first need the following proposition.

\begin{proposition}\label{levelcase}
If $R = \mathcal{R}_k[\mathcal{I}(P)]$ is $\omega^{(-1)}$-level but not Gorenstein, then we have some upwards path $(v_0,\dots,v_k)$ of length $k$ such that $(v_1, \dots, v_k)$ is a maximal length upwards path in $\hat{P}_{nonmin}$ and 
$$\sum_{i = 1}^k \disp(v_{i-1},v_i) < k.$$
\end{proposition}
  
We will prove this by contradiction.  To do so, we will need additional notation, so we leave the proof for the end of this section.

\begin{proposition}\label{onechainte}
Fix some $d_1,\dots,d_k,e \in \bbZ_{\geq 0}$, $d_1, e >0$ $p >>0$, and suppose $\sum_{i=1}^k d_i < k$ and fix $0 < e' < e$, with $e' + e'' = e$.
Then if $(N_1, \dots ,N_k) \in \mathbb{Z}^k$ satisfies:
\begin{align*}
0 \leq &N_1 \leq d_1(\pe)
\\
(*)_e\;\;\;\;\;\;\;\; 0 \leq &N_2 \leq N_1 + d_2(\pe)
\\
&\vdots 
\\
0 \leq &N_k \leq N_{k-1} + d_k(\pe)
\end{align*}
but also satisfies
\begin{equation*}
    [N_i]_{p^{e'}} > [N_{i-1}]_{p^{e'}} \text{ for all } 1 < i \leq k
\end{equation*}
then $N_i$ cannot be written $N_i = N_i''p^{e'} + N_i'$ with $N_i',N_i''$  which satisfy $(*)_{e'}$ and $(*)_{e''}$, respectively.

\end{proposition}
\begin{proof}
We assume $p > \max\{ k, d_i\}$.  Suppose
\begin{equation}\label{contradictiontosplitting}
[N_i]_{p^{e'}} > [N_{i-1}]_{p^{e'}} \text{ for all } 1 < i \leq k
\end{equation}
We will show that we will not be able to write $N_i = N_i''p^{e'} + N_i'$, for any $N_i',N_i''$ satisfying $(*)_{e'}$ and $(*)_{e''}$. 

Suppose for contradiction $N_i = N_i''p^{e'} + N_i'$, and $N_i',N_i''$ satisfy $(*)_{e'}$ and $(*)_{e''}$.

Note that, by definition, we have
\begin{align*}
    N_1' &\leq  d_1(p^{e'} -1) \\
    \implies \left\lfloor\frac{N_1'}{p^{e'}}\right\rfloor &\leq  \left\lfloor \frac{ d_1(p^{e'} -1)}{p^{e'}} \right\rfloor \\
    &= d_1 - 1
\end{align*}

Also, for all $i > 1$, we have that
\begin{align*}
    N_i' &\leq N_{i-1}' + d_i(p^{e'} -1)\\
    \implies \left\lfloor \frac{N_i'}{p^{e'}} \right\rfloor p^{e'} + [N_i']_{p^{e'}} &\leq \left\lfloor \frac{N_{i-1}'}{p^{e'}} \right\rfloor p^{e'} + [N_{i-1}']_{p^{e'}} + d_i(p^{e'} -1)   
\end{align*}
so that since $[N_i']_{p^{e'}} = [N_i]_{p^{e'}} > [N_{i-1}]_{p^{e'}} = [N_{i-1}']_{p^{e'}}$, we must have
\begin{align*}
    \left\lfloor \frac{N_i'}{p^{e'}} \right\rfloor p^{e'}  &< \left\lfloor \frac{N_{i-1}'}{p^{e'}} \right\rfloor p^{e'} + d_i(p^{e'} -1)   \\
    \implies \left\lfloor \frac{N_i'}{p^{e'}} \right\rfloor  &< \left\lfloor \frac{N_{i-1}'}{p^{e'}} \right\rfloor + \left\lfloor \frac{d_i(p^{e'} -1)}{p^{e'}}\right\rfloor \\
    &=  \left\lfloor \frac{N_{i-1}'}{p^{e'}} \right\rfloor + d_i -1.
\end{align*}

Then in particular, we have that $\left\lfloor\frac{N_1'}{p^{e'}}\right\rfloor \leq d_1 - 1$ and $\left\lfloor \frac{N_i'}{p^{e'}} \right\rfloor \leq \left\lfloor \frac{N_{i-1}'}{p^{e'}} \right\rfloor + d_i -1$ for $i > 1$ so that
\begin{equation*}
    \left\lfloor \frac{N_k'}{p^{e'}} \right\rfloor \leq \left(\sum_{i=1}^k d_i\right) - k.
\end{equation*}
This is impossible since $\sum_{i=1}^k d_i < k$, and so no such splitting can exist.
\end{proof}

\begin{theorem}\label{cxflevel}
If $\mathcal{R}_{\mathbb{F}_p}[\mathcal{I}(P)]$ is an $\omega^{(-1)}$-level Hibi ring, then
\begin{equation*}
    \lim_{p \rightarrow \infty} cx_F(\mathcal{R}_{\mathbb{F}_p}[\mathcal{I}(P)]) = |\hat{P}_{nonmin}|.
\end{equation*}
\end{theorem}

\begin{proof}
By Proposition \ref{levelcase}, we may assume that $P$ has some upwards path $(v_0,\dots,v_k)$ such that $v_0 \in \hat{P}_{min}$ and $v_i \in \hat{P}_{nonmin}$ for $i >1$ with $\sum_{i= 1}^k \disp(v_{i-1},v_i) < k$, and $(v_1, \dots, v_k)$ is a maximal length upwards path in $\hat{P}_{nonmin}$.  
By Corollary \ref{genmaps}, it suffices to show for $p >> 0$ that there are $\Theta(p^{e|\hat{P}_{nonmin}|})$ choices of maps $N: \hat{P} \rightarrow \mathbb{Z}$ 
which satisfy (\ref{one}) and (\ref{two}) for $n = \pe$ 
but cannot split as $N(v_i) = N''(v_i)p^{e'} + N'(v_i)$ where $N',N''$ satisfy (\ref{one}) and (\ref{two}) for $n = p^{e'}-1$ and $n = p^{e''} -1$, respectively.

\textbf{Lower bound:} We will construct and count $N$ which satisfy the property above.
In particular, we will construct $N$ where $N(v_i)$ cannot split for $1 \leq i \leq k$ using the method in Proposition \ref{onechainte}.

Suppose $p > \sum_{i = 1}^k \disp(v_{i-1},v_i)$.  For $0 \leq i \leq k$, we will define $N(v_i)$ by picking $0 \leq N_{i,j} \leq p-1$ for $0 \leq j \leq e$ and letting $N(v_i):= N_{i,0} + N_{i,1}p + \dots+ N_{i,e}p^e$.
Let $\mathcal{S}$ be the set of starting points as defined in Proposition \ref{orderofeq}.
Note that since $(v_1, \dots , v_k)$ is maximal in $\hat{P}_{nonmin}$, we have that for any $w$ with $v_k \lessdot w$, we know that $w \in \hat{P}_{min}$.
Define $l_v$ for every $v \in \hat{P}_{nonmin}$ as in Proposition \ref{orderofeq}: for any $v \in \hat{P}_{nonmin}$, let $l_v$ be the number of $y \in \mathcal{S}$ such that $y \leq v$, and suppose $p> 2^{l_v}$ for every $l_v$.
Note that by definition, $v_1 \in \mathcal{S}$.  For $1 \leq i \leq k$, consider $N_{i,j}$ satisfying:
\begin{align}
&0 \leq N_{1,j} < N_{2,j} < \dots < N_{k,j} \leq p-1\; \; \text{ for } 0 \leq j \leq e-2, \label{eq1} \\
&N_{i,e} = 0, \label{eq2}
\\
&\frac{p}{2^{l_{v_i} +1}} \leq N_{i,e-1} \leq \frac{(p-2^{l_{v_i}})}{2^{l_{v_i}}} \text{ for } v_i \in \mathcal{S}, \label{eq3} \text{ and }\\
&\frac{p}{2^{l_{v_i}+1}} \leq N_{i,e-1} < N_{i-1,e-1} \text{ for } v_i \notin \mathcal{S}. \label{eq4}
\end{align}
For $v \notin \{v_1, \dots, v_k\}$, consider $N(v)$ satisfying:
\begin{align}
    &N(v) = 0 \text{ if } v \in \hat{P}_{min}, \label{eq5} \\
    &\frac{(\pe)}{2^{l_v +1}} \leq N(v) \leq \frac{(\pe)}{2^{l_v}} \text{ if } v \in \mathcal{S}, \label{eq6} \text{ and }\\
    &\frac{(\pe)}{2^{l_v+1}} \leq N(v) \leq N(\tilde{v}) \text{ for } \tilde{v} \lessdot v, \text{ if }v \in \hat{P}_{nonmin} - \mathcal{S} \label{eq7} 
\end{align}
where  we note that if $v \in \hat{P}_{nonmin} -\mathcal{S}$ then $\tilde{v}$ is uniquely defined.

For $1 \leq i \leq k$, we can check that (\ref{eq1})-(\ref{eq4}) ensure 
\begin{align}
    \frac{(\pe)}{2^{l_{v_i} +1}} \leq N(v_i) &\leq \frac{(\pe)}{2^{l_{v_i}}} \text{ for } v_i \in \mathcal{S} \text{ and }\label{lb2}\\
    \frac{(\pe)}{2^{l_{v_i}+1}} \leq N(v_i) &\leq N(v_{i-1}) \text{ for } v_{i-1} \lessdot v_i, v_i \notin \mathcal{S}.\label{lb3}
\end{align}

To see this, note that for $1 \leq i \leq k$:

\begin{align*}
    N(v_i) &= N_{i,0} + \dots + N_{i,e}p^e \\
    &\geq \frac{p}{2^{l_{v_i} +1}}p^{e-1}\\
    &\geq \frac{(p^e - 1)}{2^{l_{v_i} +1}}
\end{align*}
If $v_i \in \mathcal{S}$, we also have:
\begin{align*}
    N(v_i) &= N_{i,0} + \dots + N_{i,e}p^e \\
    &\leq (p^{e-1} -1) + \frac{(p-2^{l_{v_i}})}{2^{l_{v_i}}}p^{e-1}\\
    &=\frac{p^e}{2^{l_{v_i}}} -1 \\
    &\leq \frac{(p^e - 1)}{2^{l_{v_i}}}
\end{align*}
and if  $v_i \notin S$, we have:
\begin{align*}
    N(v_i) &= N_{i,0} + \dots + N_{i,e}p^e \\
    &\leq (p^{e-1} -1) + (N_{i-1,e-1}-1)p^{e-1}\\
    &=N_{i-1,e-1}p^{e-1} - 1 \\
    &\leq N(v_{i-1})
\end{align*}
as desired.

Then in particular, by the argument in the proof of Theorem \ref{orderofeq}, all $N$ which satisfy $(\ref{eq1}) - (\ref{eq7})$ must also satisfy $(\ref{one})$ and $(\ref{two})$.  Now we must show that these $N$ also cannot split as $N(v) = N''(v)p^{e'} + N'(v)$.

We note that for all such $N$, we have $N(v_0) = 0$ since $v_0 \in \hat{P}_{min}$.  Also, since $N$ satisfies (\ref{one}) and (\ref{two}), and $v_0 \lessdot v_1 \lessdot \dots \lessdot v_k$, we have
\begin{align*}
    0 \leq &N(v_1) \leq \disp(v_0,v_1)(\pe)
\\
0 \leq &N(v_2) \leq N(v_1) + \disp(v_1,v_2)(\pe)
\\&\vdots
\\
0 \leq &N(v_k) \leq N(v_{k-1}) + \disp(v_{k-1},v_k)(\pe).
\end{align*}
We assumed $\sum_{i=1}^k \disp(v_{i-1},v_i) < k$.  Note that (\ref{eq1}) ensures
$$
[N(v_i)]_{p^{e'}} > [N(v_{i-1})]_{p^{e'}} \text{ for all } 1 < i \leq k \text{ and all } 0 < e' < e
$$
Then by Proposition \ref{onechainte}, $N$ will not split as $N(v) = N''(v)p^{e'} + N'(v)$  for any $N', N''$ since $N(v_i)$ cannot split for $1 \leq i \leq k$.

We now count how many $N$ satisfy (\ref{eq1})-(\ref{eq7}).
There are $\binom{p}{k}$ choices of $N_{1,j},\dots ,N_{k,j}$ for each $0 \leq j \leq e-2$ satisfying (\ref{eq1}) and $1$ choice of $N_{1,e},\dots,N_{k,e}$ satisfying (\ref{eq2}).  By the argument in Theorem \ref{orderofeq}, (\ref{eq3}) and (\ref{eq4}) have broken $v_1,\dots,v_k$ up into disjoint upwards paths so that $N_{i,e-1}$ satisfy inequalities of the form
$$
\frac{p}{2^{l +1}} \leq N_{i_m,e-1} <\dots< N_{i_1,e-1} \leq \frac{p-2^l}{2^l} 
$$
where $v_{i_1} \lessdot \dots \lessdot v_{i_m}$ is an upwards path with $v_{i_1} \in \mathcal{S}$, and all other $v_{i_j} \notin \mathcal{S}$.

Then in particular, for each such path of length $m$, we have 
$$\binom{\frac{p - 2^{l+1}}{2^{l + 1}}}{m} = \Theta(p^m)$$
options for $N_{i_1,e-1},\dots,N_{i_m,e-1}$.  Then in particular, we get $\Theta(p^k)$ options for $N_{1,e-1},\dots,N_{k,e-1}$ satisfying (\ref{eq3}) and (\ref{eq4}), for a total of $\Theta(p^{ke})$ options for $N(v_1),\dots,N(v_k)$ satisfying (\ref{eq1}) -(\ref{eq4}).

By the argument in the proof of Theorem \ref{orderofeq}, there are $\Theta(p^{e|\hat{P}_{nonmin}| - k})$ choices of $N(v)$ satisfying (\ref{eq5})-(\ref{eq7}) for $v \notin \{v_1,...,v_k\}$, for a total of $\Theta(p^{e|\hat{P}_{nonmin}|})$ choices of all $N$ satisfying (\ref{eq1}) - (\ref{eq7}).

\textbf{Upper bound:} An upper bound is given by Theorem \ref{orderofeq}, and again is $\Theta(p^{e|\hat{P}_{nonmin}|})$
\end{proof}

\begin{example}
We saw in Example \ref{segre} that if $R$ is the Segre product $S_{m,n}$, where $m,n \geq 2$, then $R = \mathcal{R}_k[\mathcal{I}(P)]$ where $P$ is the poset which is a chain of $n-1$ elements and a chain of $m-1$ elements.  We will see in Proposition \ref{musthavemixedpath} that Segre products of polynomial rings are $\omega^{(-1)}$-level.  Then in particular, when $m > n \geq 2$, $\hat{P}_{nonmin}$ is the chain of $m-1$ elements, and
\begin{equation*}
    \lim_{p \rightarrow \infty} \mathcal{R}_{\mathbb{F}_p}[\mathcal{I}(P)] = m-1
\end{equation*}
recovering the result of Enescu and Yao.
\end{example}

To finish the proof of our main theorem, we just need to prove Proposition \ref{levelcase}, which will require substantial setup, and which we do in the remainder of this section.

\begin{definition}
Let $p$ be a path in $P$. We say $p$ is an \textbf{upwards minimal mixed path} if $p$ is a mixed path and every upwards subpath of $p$ is a minimal length upwards path between its two endpoints.
\end{definition}

We will see that minimal generators of $\omega^{(-n)}$ with $\xi(-\infty) > -\dist(-\infty,\infty)n$ for non $\omega^{(-1)}$-level Hibi rings will correspond to upwards minimal mixed paths.  To see this, we will need the following definition:

\begin{definition}
Let $R = \mathcal{R}_k[\mathcal{I}(P)]$.  Let $\xw \in \omega^{(-n)}$, where as usual $\xi$ is a map $\xi: \hat{P} \rightarrow \mathbb{Z}$ with $\xi(\infty) = 0$.  Then, we know that we have $\xi(v_i) \geq \xi(v_j) - n$ for each $v_i \lessdot v_j \in \hat{P}$.  For each relation $v_i \lessdot v_j \in \hat{P}$ we define
\begin{equation*}
\Delta_{\w,i,j,n} :=  \xi(v_i) - (\xi(v_j) - n).
\end{equation*}
Note that $\Delta_{\w,i,j,n}  \geq 0$, and in particular, if $\xi(v_i) = \xi(v_j) - n$ for some $v_i \lessdot v_j$, we have $\Delta_{\w,i,j,n} = 0$.  When $\xi, n$ are implied, we will write $\Delta_{i,j}$ in place of $\Delta_{\w,i,j,n}$
\end{definition}

We will say $\xw \in \omega^{(-n)}$ is \textbf{compatible with} a path $p$ if $\Delta_{\w,i,j,n} = 0$ for all $v_i \lessdot v_j$ in any upwards subpath of $p$.

\begin{proposition} \label{musthavemixedpath}
Let $R =\mathcal{R}_k[\mathcal{I}(P)]$ be a Hibi ring with canonical $\omega$, and let $\xw \in \omega^{(-n)}$ such that $\xi(-\infty) > -\dist(-\infty,\infty)n$, where $\xi(\infty) = 0$.  Then $\xw$ is a minimal generator of $\omega^{(-n)}$ if and only if $\xw$ is compatible with some upwards minimal mixed path from $-\infty$ to $\infty$.
\end{proposition}
\begin{proof}
$(\leftarrow)$ Suppose $\xw \in \omega^{(-n)}$ with $\xi(-\infty) > -\dist(-\infty,\infty)n$ and $\xi(\infty) = 0$, and  suppose $\xw$ is compatible with some upwards minimal mixed path $p$ from $-\infty$ to $\infty$.  If $\xw$ is not a minimal generator of $\omega^{(-n)}$, then there must be some poset ideal $I \subset P$ such that  $\frac{\xw}{x_I} \in \omega^{(-n)}$.  Denote $\frac{\xw}{x_I} := \xwp$.  We will show that $x_I = x_P$ which will lead to a contradiction.

Note that if $v_i \lessdot v_j$ and $\Delta_{i,j}  = 0$, then if $x_i| x_I$, we must also have $x_j| x_I$, since otherwise $\xwp \not\in \omega^{(-n)}$ since it will violate the requirement $\xi'(v_i) \geq \xi'(v_j) - n$.  
We know $t| x_I$ so that if $v_i$ is in the first upwards subpath of $p$, we must have $x_i | x_I$ as well by iterating the previous argument, since $\xw$ is compatible with $p$.  Assume this subpath ends in $v_i$.  
Then for any $v_j$ in the first downwards subpath of $p$, we have $v_j \leq v_i$ so that $x_j | x_I$ since any poset ideal containing $v_i$ must also contain $v_j$.  Continuing as such, we have $x_l | x_I$ for all $v_l \in p$ so that since $p$ ends at $\infty$, we have $x_P| x_I$ and $I = P$.  
However, this is impossible.  To see this, consider $v_k \in p$ such that $v_k \lessdot \infty$.  We have assumed that $\xi(v_k) = -n$, (since $\Delta_{k,\infty} = 0$ and $\xi(\infty) = 0$) and so we cannot have $x_k | x_I$ and also have $\xwp \in \omega^{(-n)}$, since we have will violate the requirement that $\xi'(v_k) \geq -n$.

($\rightarrow$)  
Suppose $\xw$ is a minimal generator of $\omega^{(-n)}$, such that $\xi(-\infty) > -\dist(-\infty,\infty)n$ and let $\xi(\infty) = 0$.
Let $\mathscr{A}$ be the family of poset ideals $I \subset \hat{P}$ such that there exists a choice of generators $I = \left<v_{I,i}\right>$ of $I$ such that for each $i$, there exists a mixed or upwards path $p$ starting at $-\infty$ and ending at $v_{I,i}$, and $\xw$ is compatible with $p$.  We show that $\infty \in I$ for some $I \in \mathscr{A}$. Then because $\infty$ is maximal, it must be a generator of $I$ and there is a mixed or upwards path compatible with $\xw$ from $-\infty$ to $\infty$.  

Suppose for contradiction that $\infty \notin I$ for any $I \in \mathscr{A}$.  Let $I$ be a maximal element in $\mathscr{A}$ with respect to inclusion, and let $\{v_{I,i}\}$ be the choice of generators of $I$ with the desired property.  Since $\xw$ is a minimal generator of $\omega^{(-n)}$ over $R$, we know $\frac{\xw}{x_I} \notin \omega^{(-n)}$.  Then we must have some $v_j \lessdot v_k$ with $v_j \in I$, $v_k \notin I$ and $\Delta_{j,k} = 0$ (since otherwise $\Delta_{j,k} > 0$ for all such $v_j, v_k$ and so $\frac{\xw}{x_I} \in \omega^{(-n)}$).  Since $v_j \in I$, we must have $v_j \leq v_{I,i}$ for some $i$.
\\
Case 1: $v_j = v_{I,i}$.  In this case, by construction we have a mixed or upwards path $p$ starting at $-\infty$ and ending at $v_j$ with $\Delta_{r,s}=0$ along its upward subpaths.  If $v_k \in p$, we have a mixed path starting at $-\infty$ and ending at $v_k$ with $\Delta_{r,s}  = 0$ along upwards subpaths by just truncating $p$.  Otherwise, if $v_k \notin p$, then we can extend $p$ by adding the upwards segment between $v_j$ and $v_k$ to get a mixed or upwards path starting at $-\infty$ and ending at $v_k$ such that along all its upward subpaths, $\Delta_{r,s}  = 0$.  Then the ideal generated by $\{v_{I,i}\} \cup v_k$ is in $\mathscr{A}$ and contains $I$, giving us a contradiction since we assumed $I$ was maximal.
\\
Case 2: $v_j < v_{I,i}$.  Again, we have a mixed or upwards path $p$ starting at $-\infty$ and ending at $v_{I,i}$ such that along upwards subpaths, $\Delta_{r,s}  = 0$.  If $v_k \in p$, we truncate $p$ as before to get such a path ending instead at $v_k$.  Likewise, if $v_j \in p$, we truncate $p$ at $v_j$ and then add the upwards subpath between $v_j$ and $v_k$.  Otherwise, if $v_j,v_k \notin p$, pick $v_l$ such that $v_l \in p$ and $v_j < v_l$, and such that $v_l$ is minimal with this property (so that there is no $v_s \in p$ with $v_j < v_s < v_l$).  Then the downwards subpath between $v_l$ and $v_j$ is not in $p$ by construction so we truncate $p$ at $v_l$, add this downwards subpath, and then add the upwards subpath from $v_j$ to $v_k$. This produces another mixed path starting at $-\infty$ and ending at $v_k$ which is compatible with $\xw$.  Then again, the ideal generated by $\{v_{I,i}\} \cup v_k$ is in $\mathscr{A}$ and contains $I$, giving us a contradiction.

Now we know there is a mixed or upwards path with the desired property starting at $-\infty$ and ending at $\infty$.  We need to show that $p$ is mixed.  If $p$ is an upwards path, then since $\xw$ is compatible with $p$, we have $\xi(-\infty) = -\len(p)n$.  In particular, since we must have $\len(p) \geq \dist(-\infty,\infty)$, we have $\xi(-\infty) \leq -\dist(-\infty,\infty)n$ so that since $\xi(-\infty) \geq -\dist(-\infty,\infty)n$ we have $\xi(-\infty) = -\dist(-\infty,\infty)n$. Since we assumed $\xi(-\infty) > -\dist(-\infty,\infty)n$ this is impossible, so in fact $p$ must be mixed.

Finally, we need to show that $p$ is upwards minimal.  Consider some upwards segment $u$ of $p$ starting at $v_i$ and ending at $v_j$.
Note that by iterating $\xi(v_i) \geq \xi(v_k) -n$ for $v_i \lessdot v_k$ along a minimal upwards path between $v_i$ and $v_j$, we get that 
\begin{equation*}
    \xi(v_i) \geq \xi(v_j) - \dist(v_i,v_j)n
\end{equation*}
If we have $\Delta_{r,s}  = 0$ for all $v_r \lessdot v_s$ along $u$, we get that
\begin{equation*}
    \xi(v_i) = \xi(v_j) - \len(u)n.
\end{equation*}
Then, in particular, we must have
\begin{equation*}
    \len(u) \leq \dist(v_i,v_j),
\end{equation*}
but since $\dist(v_i,v_j)$ is minimal, we have that
\begin{equation*}
    \len(u) = \dist(v_i,v_j).
\end{equation*}
Thus every upwards segment of $p$ must have minimal length.
\end{proof}

It will be useful to define an ordering on $p$--that is a notion of which vertex ``comes first" when traveling from $-\infty$ to $\infty$ along $p$.  We make the following definition:

\begin{definition}
Fix a mixed path $p$, and fix $v_i \neq v_j \in p$ (recall by definition $p$ has no self intersections).  We say
\begin{equation*}
    v_i \prec_p v_j
\end{equation*}
if $v_i$ comes before $v_j$ when we travel along $p$, i.e., $p = (\dots v_i \dots v_j \dots)$.  Note that it is not necessary that $v_i \leq v_j$.
\end{definition}

\begin{definition}
Fix a mixed path $p$ and points $v_i \prec_p v_j$ in $p$.  Let $\{u_k\}$ be the set of upwards subpaths of $p$ between $v_i$ and $v_j$ along $p$ and let $\{d_l\}$ be the set of downwards subpaths of $p$ between $v_i$ and $v_j$ along $p$.  We define the \textbf{mixed length} of $p$ between $v_i$ and $v_j \in p$ denoted $\mlen_p(v_i,v_j)$ to be
\begin{equation*}
    \mlen_p(v_i,v_j) := \sum_k \len(u_k) - \sum_l \len(d_l)
\end{equation*}
\end{definition}

We are now ready to prove Proposition \ref{levelcase}, which we restate here for convenience:

\begin{prop:levelcase}
If $R = \mathcal{R}_k[\mathcal{I}(P)]$ is $\omega^{(-1)}$-level but not Gorenstein, then we have some upwards path $(v_0,\dots,v_k)$ of length $k$ such that $(v_1, \dots, v_k)$ is a maximal length upwards path in $\hat{P}_{nonmin}$ and 
$$\sum_{i = 1}^k \disp(v_{i-1},v_i) < k.$$
\end{prop:levelcase}
\begin{proof}
Suppose $R= \mathcal{R}_k[\mathcal{I}(P)]$ is $\omega^{(-1)}$-level and suppose for contradiction that no maximal length upwards path in $\hat{P}_{nonmin}$ satisfies the property above.

Pick a maximal length upwards path $p = (v_1,\dots,v_k)$ in $\hat{P}_{nonmin}$,
(which we can do since $R$ is not Gorenstein).  Pick any $v_0 \lessdot v_1$.  
Since we picked $p$ to be maximal, we must have that $v_0 \in \hat{P}_{min}$. Then $ (v_0,\dots,v_k)$ is an upwards path of length $k$, and we have assumed that
$$\sum_{i=1}^k \disp(v_{i-1},v_i) \geq k$$ 
Note that
\begin{align*}
    \sum_{i=1}^k \disp(v_{i-1}, v_i)
    &= \sum_{i=1}^k \dist(v_{i-1},v_i) + \dist(v_i,\infty) - \dist(v_{i-1},\infty) \\
    &= \sum_{i=1}^k 1 + \dist(v_i,\infty) - \dist(v_{i-1},\infty) \\
    &= \dist(v_k,\infty) - \dist(v_0,\infty) + k,
\end{align*}
so that 
\begin{equation*}
    \sum_{i=1}^k \dist(v_{i-1},v_i) \geq k
    \implies \dist(v_k,\infty) \geq \dist(v_0,\infty).
\end{equation*}

We will show that this will give us that $R$ is not $\omega^{(-1)}$-level, by constructing an upwards minimal mixed path and explicitly writing down a minimal generator $\xw$ of $\omega^{(-1)}$ which is compatible with it and has $\xi(-\infty) > -\dist(-\infty,\infty)$ where $\xi(\infty) = 0$.  First, note that we cannot have $v_k \lessdot \infty$, since then $\dist(v_k,\infty) = 1$ but since $v_0 \leq v_k \leq \infty$, we must have $\dist(v_0,\infty) > 1$ and so using the equation above we get $\dist(v_k,\infty) > 1$.

Pick an element $v_{k+1}$ covering $v_k$ which is on  a minimal path between $v_k$ and $\infty$, so that $v_k \lessdot v_{k+1}$ and $\dist(v_k,\infty) = \dist(v_{k+1},\infty) + 1$.
Note that $v_{k+1} \not= \infty$ by the previous paragraph. 
Consider a minimal length upwards path from $-\infty$ to $v_{k+1}$, and call this $u_1$.  
Then pick a minimal length upwards path from $v_0$ to $\infty$ and call this $u_2$.  
Let $d$ be the downwards path that goes from $v_k$ to $v_1$ along $p$.  
We will show that $q = u_1 \cup d \cup u_2$ is an upwards minimal mixed path which will correspond to a minimal generator $\xw$ of $\omega^{(-1)}$ with $\xi(-\infty) > -\dist(-\infty,\infty)$.

First, we show that $u_1$ and $u_2$ will not intersect.  Suppose they do.  Break them up into segments $u_1 = u_1' \cup u \cup u_1''$ and $u_2 = u_2'' \cup u \cup u_2'$ as pictured, and let $w$ be the largest point in $u$, i.e. the largest point in both $u_1$ and $u_2$.
Note that if $u_1$ and $u_2$ were to intersect in a disjoint set of elements, their lengths between their intersection points would be the same by construction, so that we could instead consider $u_1$ and $u_2$ that intersect in one connected path as pictured.

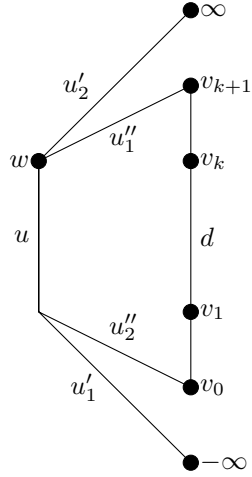
\begin{center}
\begin{tikzpicture}
\draw (0,0)-- (-2,-2) -- (-2,-4)-- (0,-5)-- (0,-1) -- (-2,-2) -- (-2,-4) -- (0,-6);
\draw[fill] (0,0) circle [radius=0.1];
\draw[fill] (0,-1) circle [radius=0.1];
\draw[fill] (0,-2) circle [radius=0.1];
\draw[fill] (0,-4) circle [radius=0.1];
\draw[fill] (0,-5) circle [radius=0.1];
\draw[fill] (0,-6) circle [radius=0.1];
\draw[fill] (-2,-2) circle [radius=0.1];
\node [right] at (0,-6) {$-\infty$};
\node [right] at (0,-5) {$v_0$};
\node [right] at (0,-4) {$v_1$};
\node [right] at (0,-2) {$v_k$};
\node [right] at (0,-1) {$v_{k+1}$};
\node [right] at (0,0) {$\infty$};
\node [left] at (-1.2,-1) {$u_2'$};
\node [left] at (-2,-3) {$u$};
\node [right] at (-1.2,-4.2) {$u_2''$};
\node [left] at (-1.1,-5) {$u_1'$};
\node [right] at (-1.2,-1.7) {$u_1''$};
\node [right] at (0,-3) {$d$};
\node [left] at (-2,-2) {$w$};
\end{tikzpicture}
\end{center}
Since $w$ is on a minimal length path from $-\infty$ to $v_{k+1}$, and $v_{k+1} \in \hat{P}_{min}$, we have:
\begin{equation*}
    \len(u_1'') + \dist(v_{k+1},\infty) = \dist(w,\infty) = \len(u_2')
\end{equation*}
Then in particular,
\begin{align*}
    \dist(v_0,\infty) &= \len(u_2'') + \len(u) + \len(u_2') \\
    &= \len(u_2'') + \len(u) + \len(u_1'') + \dist(v_{k+1},\infty),
\end{align*}
and since $u_2'' \cup u \cup u_1''$ is a path between distinct points, it must have at least length $1$.  Then we have
\begin{equation*}
    \dist(v_0,\infty) \geq \dist(v_{k+1},\infty) +1 = \dist(v_k,\infty).
\end{equation*}
Then by our assumption we must have $\dist(v_k,\infty) = \dist(v_0,\infty)$, so that
\begin{align*}
    \dist(v_{k+1},\infty)+1 
    &= \dist(v_0,\infty) \\
    &= \len(u_2'') + \len(u) + \len(u_1'') + \dist(v_{k+1},\infty) 
\end{align*}
and $\len(u_2'') + \len(u) + \len(u_1'') = 1$.  However, $u_2'' \cup u \cup u_1''$ is a path from $v_0$ to $v_{k+1}$, so this tells us that $v_0 \lessdot v_{k+1}$, which is impossible since there is already a path of length $k+1$ between $v_0$ and $v_{k+1}$.  Thus, we can assume that $q$ does not self intersect.  It is already clear that $u_1$ and $u_2$ cannot intersect $d$, since $d$ lies entirely in $\hat{P}_{nonmin}$, and $u_1,u_2$ lie entirely in $\hat{P}_{min}$.

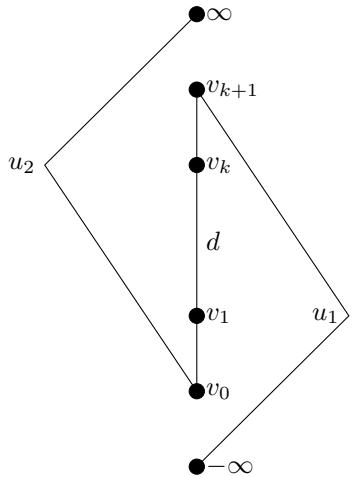
\begin{center}
\begin{tikzpicture}
\draw (0,0)-- (-2,-2) -- (0,-5)-- (0,-1) -- (2,-4) -- (0,-6);
\draw[fill] (0,0) circle [radius=0.1];
\draw[fill] (0,-1) circle [radius=0.1];
\draw[fill] (0,-2) circle [radius=0.1];
\draw[fill] (0,-4) circle [radius=0.1];
\draw[fill] (0,-5) circle [radius=0.1];
\draw[fill] (0,-6) circle [radius=0.1];
\node [right] at (0,-6) {$-\infty$};
\node [right] at (0,-5) {$v_0$};
\node [right] at (0,-4) {$v_1$};
\node [right] at (0,-2) {$v_k$};
\node [right] at (0,-1) {$v_{k+1}$};
\node [right] at (0,0) {$\infty$};
\node [left] at (-2,-2) {$u_2$};
\node [left] at (2,-4) {$u_1$};
\node [right] at (0,-3) {$d$};
\end{tikzpicture}
\end{center}

Then our mixed path $q$ looks like the picture above, and by construction it is upwards minimal.  We will explicitly construct $\xw \in \omega^{(-1)}$ compatible with $q$ (so that it is a minimal generator of $\omega^{(-1)}$), which has $\xi(-\infty) > -\dist(-\infty,\infty)$.  Define $\xw$ by:
\begin{align*}
    \xi(y) &= -\dist(y,\infty) = -\mlen_q(y,\infty) \text{ for } y \in u_2, \\
    \xi(v_{k+1}) &= \min\{\dist(y,z) + \dist(z,v_{k+1})-\mlen_q(y,\infty)  | y < z \leq v_{k+1}, y,z \in q, z \prec_q y\}, \\
    \xi(y) &= \max\{ -\dist(y,\infty), \xi(v_{k+1}) -\dist(y,v_{k+1}) \} \text{ for } y \in d, \\
    \xi(y) &=  \xi(v_{k+1}) -\dist(y,v_{k+1})\text{ for } y \in u_1, y \neq v_{k+1}, \text{ and} 
    \\
    \xi(y) &= \min\{ \xi(x) + \dist(x,y) | x \leq y, x \in q \} \text{ for } y \notin q.
\end{align*}

We first show that $\xi(v_{k+1}) > -\dist(v_{k+1},\infty)$.

By definition, $\xi(v_{k+1}) = \dist(y,z) + \dist(z,v_{k+1})-\mlen_q(y,\infty)$ for some $y,z$ such that $y < z \leq v_{k+1}$ and $z \prec_p y$. 
Note that this set of $y,z$ is nonempty since we have $v_0 \leq v_k \leq v_{k+1}$ and $v_k \prec_p v_0$. 
We cannot have $y \in u_1$ since $u_1$ is an upwards path, so  we have either $y \in u_2$ or $y \in d$. 

If $y \in u_2$ and $z \neq v_{k+1}$ then:
\begin{align*}
    \xi(v_{k+1}) &= \dist(y,z) + \dist(z,v_{k+1}) -\dist(y,\infty)\\
    &\geq  \dist(y,z) + \dist(z,v_{k+1})-\dist(v_0,\infty)\\
    &\geq \dist(y,z) + \dist(z,v_{k+1})-\dist(v_k,\infty)\\
    &>1-\dist(v_k,\infty) \\
    &= -\dist(v_{k+1},\infty)
\end{align*}

If $y \in u_2$ and $z = v_{k+1}$ then:
\begin{align*}
    \xi(v_{k+1}) &= \dist(y,z)-\dist(y,\infty) \\
    &\geq  \dist(y,z) -\dist(v_0,\infty)\\
    &\geq \dist(y,z) -\dist(v_k,\infty)\\
    &\geq  1 -\dist(v_k,\infty)\\
    &= -\dist(v_{k+1},\infty)
\end{align*}
but either the first inequality is strict if $y \neq v_0$, or the third inequality is strict if $y = v_0$, and either way $ \xi(v_{k+1}) > -\dist(v_{k+1},\infty)$.

Otherwise, if $y \in d$ then
\begin{align*}
    \xi(v_{k+1}) &=\dist(y,z) + \dist(z,v_{k+1}) -\dist(v_0,\infty) -\mlen_q(y,v_0)\\
    &> \dist(y,z) + \dist(z,v_{k+1})-\dist(v_0,\infty)\\
    &\geq \dist(y,z) + \dist(z,v_{k+1})-\dist(v_k,\infty)\\
    &\geq  1 -\dist(v_k,\infty)\\
   & = -\dist(v_{k+1},\infty).
\end{align*}

Now we have shown that $\xi(v_{k+1}) > -\dist(v_{k+1},\infty)$.  Then in particular, we have
\begin{align*}
\xi(-\infty) &= \xi(v_{k+1}) - \dist(-\infty,v_{k+1}) \\
&> -\dist(v_{k+1},\infty) - \dist(-\infty,v_{k+1}) \\
&= -\dist(-\infty,\infty)
\end{align*}
as desired.  
Now we have checked that $\xi(-\infty) > -\dist(-\infty,\infty)$.
In particular, if we can show $\xw$ is a minimal generator of $\omega^{(-1)}$, this tells us that $R$ is not $\omega^{(-1)}$ level.  
It suffices to check that $\xw \in \omega^{(-1)}$ and that $\xw$ is compatible with $q$, (since by definition $\xi(\infty) = -\dist(\infty,\infty) = 0$).

By construction of $\xw$, $\xw$ is compatible with $q$, so that $\xw$ is a minimal generator of $\omega^{(-1)}$ as long as $\xw \in \omega^{(-1)}$ by Proposition \ref{musthavemixedpath}.  
To check that $\xw \in \omega^{(-1)}$ we must check that for every $y \lessdot z \in \hat{P}$, we have $\xi(y) \geq \xi(z) -1$. 

First, we check this for $y,z \in q$.  We will actually show the stronger statement: $\xi(y) - \xi(z) \geq -\dist(y,z)$ for any $y \leq z$ with $y,z \in q$.

Note that by our definition, if $y \in q$ we have either
\begin{align*}
    \xi(y) &= \xi(v_{k+1}) - \dist(y,v_{k+1}) \text{ or} \\
    \xi(y) &= -\dist(y,\infty).
\end{align*}
Note the first of these equations holds trivially for $y = v_{k+1}$.  Fix $y,z \in q$ with $y \leq z$.   

Note that if $\xi(y) \geq -\dist(y,\infty)$ and $\xi(z) = -\dist(z,\infty)$ we have

\begin{align*}
    \xi(y) - \xi(z) &\geq -\dist(y,\infty) +\dist(z,\infty)\\
    &\geq -\dist(y,z)
\end{align*}
by (\ref{triangle}).

Similarly, if $\xi(y) \geq \xi(v_{k+1}) - \dist(y,v_{k+1})$ and $\xi(z) = \xi(v_{k+1}) - \dist(z,v_{k+1})$, we have:

\begin{align*}
    \xi(y) - \xi(z) &\geq \xi(v_{k+1}) - \dist(y,v_{k+1}) -\xi(v_{k+1}) + \dist(z,v_{k+1}) \\
    &= -\dist(y,v_{k+1}) + \dist(z,v_{k+1})\\
    &\geq -\dist(y,z)
\end{align*}
where again, the last equation holds by (\ref{triangle}).

We have four cases to check.

\textbf{Case 1}:  If $y$ and $z$ and both are in $u_1$ or both are in $u_2$, then by definition $\xi(y) = \xi(z) -\dist(y,z)$.

\textbf{Case 2:}
If $y \in u_2$, and $z \in d$ or $u_1$, we have $y \leq z$ and $z \prec_q y$ so that
\begin{align*}
    \xi(v_{k+1}) &\leq \dist(y,z) + \dist(z,v_{k+1}) -\mlen_q(y,\infty)
    \\
    &=\dist(y,z) + \dist(z,v_{k+1})-\dist(y,\infty).
\end{align*}

Then if in addition $\xi(z) = \xi(v_{k+1}) - \dist(z,v_{k+1})$, we have
\begin{align*}
    \xi(y) - \xi(z) &= -\dist(y,\infty) -\xi(v_{k+1}) + \dist(z,v_{k+1})\\
    & \geq -\dist(y,\infty) - \dist(y,z) - \dist(z,v_{k+1})+\dist(y,\infty) + \dist(z,v_{k+1})\\
    &= -\dist(y,z).
\end{align*}

We have already checked the case that $\xi(z) = -\dist(z,\infty)$, since $y \in u_2 \implies \xi(y) = -\dist(y,\infty)$.

\textbf{Case 3:}
Suppose $y \in d$ and $z \in u_1,d,$ or $u_2$.  Note that we have checked the case that $\xi(z) = \xi(v_{k+1}) - \dist(z,v_{k+1})$, since $y \in d \implies \xi(y) \geq \xi(v_{k+1})- \dist(y,v_{k+1})$.  Similarly, we have checked the case that  $\xi(z) = -\dist(z,\infty)$, since $y \in d \implies \xi(y) \geq -\dist(y,\infty)$.

\textbf{Case 4:}
Suppose $y \in u_1$ and  $z \in u_2$ or $d$.  We have checked the case that $\xi(z) = \xi(v_{k+1}) -\dist(z,v_{k+1})$, since $y \in u_1 \implies \xi(y) = \xi(v_{k+1}) - \dist(y,v_{k+1})$.  Otherwise, if $\xi(z) = -\dist(z,\infty)$  we have:
\begin{align*}
    \xi(y) - \xi(z) &= \xi(v_{k+1}) - \dist(y,v_{k+1}) + \dist(z,\infty) \\
    &> -\dist(v_{k+1},\infty)  - \dist(y,v_{k+1}) + \dist(z,\infty) \\
    &=-\dist(y,\infty) + \dist(z,\infty) \\
    &\geq -\dist(y,z)
\end{align*}
where the second equality holds because $y$ is on a minimal length path between $-\infty$ and $v_{k+1} \in \hat{P}_{min}$, and the last inequality holds by (\ref{triangle}).

We have checked that anytime $y \leq z$ with $y,z \in q$, we have $\xi(y) \geq \xi(z) - \dist(y,z)$.  In particular, if $y \lessdot z$ we have $\xi(y) \geq \xi(z) -1$ as desired.

Now suppose $y \lessdot z$ with $y, z \notin q$.  Then if $\xi(y) = \xi(x) + \dist(x,y)$ for some $x \in q$, we have $x \leq z$ so that
\begin{align*}
 \xi(z) &\leq \xi(x) + \dist(x,z)
 \\
 &\leq \xi(x) + \dist(x,y) + 1 \\
 &= \xi(y) + 1.   
\end{align*}
Similarly, if $y \lessdot z$ with $z \notin q, y \in q$, then we have
\begin{align*}
 \xi(z) &\leq \xi(y) + \dist(y,z)
 \\
 &= \xi(y) + 1  
\end{align*}
so that $\xi(y) \geq \xi(z) - 1$, as desired.

Finally, suppose $y \notin q, z \in q$.  We have
\begin{equation*}
    \xi(y) = \min\{ \xi(x) + \dist(x,y) | x < y, x \in q\}
\end{equation*}
(in particular, there is at least one such $x$ since $-\infty \in q$).  Say $\xi(y) = \xi(x) + \dist(x,y)$ for some $x \in q$.  Then we have
\begin{align*}
    \xi(y) - \xi(z) &= \xi(x) + \dist(x,y) - \xi(z) \\
    &\geq -\dist(x,z) + \dist(x,y) \\
    &\geq -\dist(y,z) \\
    &= -1,
\end{align*}
as desired, where the first inequality holds because $x,z \in q$ and $x \leq z \implies  \xi(x) - \xi(z) \geq -\dist(x,z)$, and the second inequality holds by (\ref{triangle}).  Now we have constructed $\xw \in \omega^{(-1)}$ such that $\xi(y) = \xi(z) -1$ for all $y \lessdot z$ along $u_1,u_2$ so that by Proposition \ref{musthavemixedpath}, $\xw$ is a minimal generator of $\omega^{(-1)}$.  Also, we have ensured $\xi(-\infty) > -\dist(-\infty,\infty)$ so that $R$ cannot be $\omega^{(-1)}$-level.
\end{proof}

\section{Frobenius complexity for non anticanonical level Hibi rings}
We end with a few comments on the non $\omega^{(-1)}$-level case, and we give an example.  We saw in Proposition \ref{musthavemixedpath} that minimal generators $\xw$ of $\omega^{(-n)}$ which do not have minimal degree (on $t$) correspond to upwards minimal mixed paths.  Inequalities describing these generators can be somewhat more cumbersome to write out in full generality, but similar methods for counting these generators and computing the Frobenius complexity are often sufficient. We demonstrate the method for  Example \ref{levelex} (1), and draw $\hat{P}$ below.
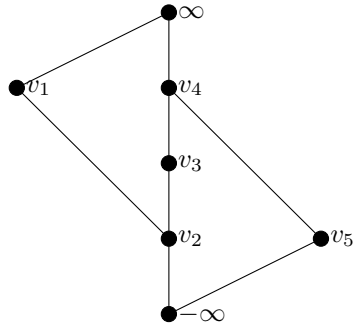
\begin{center}
\begin{tikzpicture}
\draw (0,1) -- (-2,0) -- (0,-2) -- (0,0) -- (2,-2) -- (0,-3);
\draw (0,1) -- (0,0);
\draw (0,-2) -- (0,-3);
\draw[fill] (0,0) circle [radius=0.1];
\draw[fill] (0,-1) circle [radius=0.1];
\draw[fill] (0,-2) circle [radius=0.1];
\draw[fill] (-2,0) circle [radius=0.1];
\draw[fill] (2,-2) circle [radius=0.1];
\draw[fill] (0,1) circle [radius=0.1];
\draw[fill] (0,-3) circle [radius=0.1];

\node [right] at (0,0) {$v_4$};
\node [right] at (0,-1) {$v_3$};
\node [right] at (0,-2) {$v_2$};
\node [right] at (-2,0) {$v_1$};
\node [right] at (2,-2) {$v_5$};
\node [right] at (0,1) {$\infty$};
\node [right] at (0,-3) {$-\infty$};
\end{tikzpicture}
\end{center}
Let $R =\mathcal{R}_{\mathbb{F}_p}[\mathcal{I}(P)]$.  Minimal generators of $\omega^{(-n)}$ correspond to upwards minimal mixed paths between $-\infty$ and $\infty$, and here there is exactly one: $p = (-\infty, v_5, v_4, v_3, v_2, v_1, \infty)$.  Let $\xw$ be a minimal generator of $\omega^{(-n)}$.  Since $\xi(\infty) = 0$, we have:
\begin{align*}
    \xi(v_1) &= -n \text{ and}\\
    \xi(v_2) &= -2n.
\end{align*}
Similarly, if we fix $\xi(v_4)$, we have:
\begin{align*}
    \xi(v_5) &= \xi(v_4) -n \text{ and} \\
    \xi(-\infty) &= \xi(v_5) - n,
\end{align*}
so it suffices to determine all possible values of $\xi(v_3)$ and $\xi(v_4)$.  We have:
\begin{align*}
    -2n &\leq \xi(v_3) \leq -n \\
    -n &\leq \xi(v_4) \leq \xi(v_3) + n\\
\end{align*}
Writing $N(v_3) := \xi(v_3) + 2n$ and $N(v_4) := \xi(v_4) + n$, we have:
\begin{align*}
    0 &\leq N(v_3) \leq n \text{ and}\\
    0 &\leq N(v_4) \leq N(v_3),
\end{align*}
and so we can see that $sp_R(\omega^{(-1)})-1 = 2$.  Also, if $n = p^{e}-1$, we have that by Proposition \ref{onechainte}, $\xw$ cannot split as $\xwpp*\xwp$ as long as $[N(v_4)]_{p^{e'}} > [N(v_3)]_{p^{e'}}$ for all $0 < e' < e$.  In particular, we can let
\begin{align*}
    0 &< N_{4,j} < N_{3,j} < p \text{ for } 0 \leq j \leq e-2 \\
    0 &\leq N_{3,e-1} < N_{4,e-1} < p \\
    &N_{3,e} = N_{4,e} = 0
\end{align*}
and let $N(v_i) := N_{i,0} + N_{i,1}p + \dots + N_{i,e}p^e$ for $i = 3,4$.  Then $\xw$ will not split as $\xwpp*\xwp$ for any $e' + e'' = e$, so that there are at least $\binom{p}{2}^e = \Theta(p^{2e})$ minimal generators $\xw$ of $\omega^{(1-p^e)}$ that cannot split as smaller degree terms and $\lim_{p \rightarrow \infty} cx_F(R) = 2$.  

In this case, we call $v_4$ a \textbf{peak} in $p$, as it is the last element in an upwards subpath of $p$ and the first element in a downwards subpath in $p$.  In general, for non $\omega^{(-1)}$ level Hibi rings, there will be some upwards minimal mixed path $p$ between $-\infty$ and $\infty$ in $\hat{P}$, and  $sp_R(\omega^{(-1)})-1$ will tell us the number of elements in $\hat{P}$ which are either peaks in $p$ or are not in upwards subpaths of $p$.  In every case we have computed, $\lim_{p \rightarrow \infty} cx_F(R)= sp_R(\omega^{(-1)})-1$, 
but we do not yet have an analogue of Proposition \ref{levelcase} to say that these numbers must always match up.  This leads us to the following question, which we hope to answer in future work.

\begin{question}
Given a Hibi ring $R$ (or more generally, any toric ring) over a field of characteristic $p$, is it always true that $\lim_{p \rightarrow \infty} cx_F(R)= sp_R(\omega_R^{(-1)})-1$?
\end{question}

\bibliographystyle{elsarticle-harv} 
\bibliography{bibliography}{}

\begin{thebibliography}{16}
\expandafter\ifx\csname natexlab\endcsname\relax\def\natexlab#1{#1}\fi
\expandafter\ifx\csname url\endcsname\relax
  \def\url#1{\texttt{#1}}\fi
\expandafter\ifx\csname urlprefix\endcsname\relax\def\urlprefix{URL }\fi

\bibitem[{{\`Alvarez Montaner} et~al.(2012){\`Alvarez Montaner}, Boix, and
  Zarzuela}]{MBZFrobandCartAlgofSRRings}
{\`Alvarez Montaner}, J., Boix, A.~F., Zarzuela, S., 2012. Frobenius and
  {C}artier algebras of {S}tanley-{R}eisner rings. Journal of Algebra 358,
  162--177.

\bibitem[{Blickle(2013)}]{BTestIdealsofpeLinearMaps}
Blickle, M., 2013. Test ideals via algebras of $p^{-e}$-linear maps. Journal of
  Algebraic Geometry 22~(1), 49--83.

\bibitem[{Bruns and Gubeladze(2009)}]{bruns2009polytopes}
Bruns, W., Gubeladze, J., 2009. Polytopes, Rings, and K-Theory. Springer
  Monographs in Mathematics. Springer New York.

\bibitem[{Enescu and Yao(2016)}]{EYTheFrobcompofalocalringofprimechar}
Enescu, F., Yao, Y., 2016. The {F}robenius complexity of a local ring of prime
  characteristic. Journal of Algebra 459, 133--156.

\bibitem[{Enescu and Yao(2018)}]{EYOntheFrobcompofdetrings}
Enescu, F., Yao, Y., 2018. On the {F}robenius complexity of determinantal
  rings. J. Pure Appl. Algebra 222~(2), 414--432.

\bibitem[{Fulton(1993)}]{FIntrotoToricVarieties}
Fulton, W., 1993. {I}ntroduction to {T}oric {V}arieties. Princeton University
  Press.

\bibitem[{Hibi(1987)}]{HDisplat}
Hibi, T., 1987. Distributive lattices, affine semigroup rings and algebras with
  straightening laws. In: Commutative Algebra and Combinatorics (M. Nagata and
  H. Matsumura, eds.). Vol.~11. Advanced Studies in Pure Math., North-Holland,
  Amsterdam, pp. 93--109.

\bibitem[{Katzman(2010)}]{KAnExofNonfingenalgofFrobMaps}
Katzman, M., 2010. A non-finitely generated algebra of {F}robenius maps. Proc.
  Amer. Math. Soc. 138, 2381--2383.

\bibitem[{Katzman et~al.(2014)Katzman, Schwede, Singh, and
  Zhang}]{KSSZRingsofFrobOp}
Katzman, M., Schwede, K., Singh, A.~K., Zhang, W., 2014. Rings of {F}robenius
  operators. Math. Proc. Camb. Phil. Soc. 157~(1), 151--167.

\bibitem[{Lyubeznik and Smith(2001)}]{LSTestIdeal}
Lyubeznik, G., Smith, K., 2001. On the commutation of the test ideal with
  localization and completion. Trans. Amer. Math. Soc 353, 3149--3180.

\bibitem[{Mehta and Ramanathan(1985)}]{MRFrobeniussplittingandcohovanishing}
Mehta, V.~B., Ramanathan, A., 1985. Frobenius splitting and cohomology
  vanishing for {S}chubert varieties. Annals of Mathematics 122~(1), 27--40.

\bibitem[{Miyazaki(2007)}]{MASuffCondforaHibiRingtobeLevel}
Miyazaki, M., 2007. A sufficient condition for a {H}ibi ring to be level and
  levelness of {S}chubert cycles. Comm. Algebra 35, 2894--2900.

\bibitem[{Miyazaki(2017)}]{MOnGenCanModuleHibiRing}
Miyazaki, M., 2017. On the generators of the canonical module of a {H}ibi ring:
  a criterion of level property and the degrees of generators. Journal of
  Algebra 480, 215--236.

\bibitem[{Schwede(2011)}]{STestIdealsinNonQGorRings}
Schwede, K., 2011. Test ideals in non-{Q}-{G}orenstein rings. Transactions of
  the American Mathematical Society 363~(11), 5925--5941.

\bibitem[{Smith and Zhang(2015)}]{SZFrobspilttingincomalg}
Smith, K.~E., Zhang, W., 2015. Frobenius splitting in commutative algebra. In:
  Commutative algebra and noncommutative algebraic geometry. {V}ol. {I}.
  Vol.~67 of Math. Sci. Res. Inst. Publ. Cambridge Univ. Press, New York, pp.
  291--345.

\bibitem[{Stanley(1977)}]{Scmcomplexes}
Stanley, R.~P., 1977. Cohen-{M}acaulay complexes. In: Higher Combinatorics.
  Vol.~31 of NATO Advanced Study Institute Series. Springer Netherlands, pp.
  51--62.

\end{thebibliography}

\end{document}